\documentclass[10pt,a4paper]{amsart}
\usepackage{amsmath,amsthm}
\usepackage{amssymb}
\usepackage{comment}
\DeclareMathOperator{\image}{''}
\DeclareMathOperator{\acc}{acc}
\DeclareMathOperator{\dom}{dom}
\DeclareMathOperator{\otp}{otp}
\DeclareMathOperator{\supp}{supp}
\DeclareMathOperator{\len}{len}
\DeclareMathOperator{\crit}{crit}
\DeclareMathOperator{\im}{im}

\DeclareMathOperator{\cf}{cf}
\DeclareMathOperator{\stem}{stem}
\DeclareMathOperator{\Add}{Add}
\DeclareMathOperator{\Col}{Col}

\DeclareMathOperator{\Refl}{Refl}
\DeclareMathOperator{\power}{\mathcal{P}}
\newcommand{\ZFC}{{\rm ZFC}}
\newcommand{\GCH}{{\rm GCH}}
\newcommand{\Ord}{{\rm Ord}}

\newtheorem{theorem}{Theorem}

\newtheorem{fact}[theorem]{Fact}
\newtheorem{remark}[theorem]{Remark}

\newtheorem{lemma}[theorem]{Lemma}
\newtheorem{claim}[theorem]{Claim}
\theoremstyle{definition}
\newtheorem{definition}[theorem]{Definition}
\newtheorem{question}{Question}

\newtheorem*{theorem*}{Theorem}

\title{Stationary Reflection}
\author{Yair Hayut}
\email[Yair Hayut]{yair.hayut@mail.tau.ac.il}
\address[Yair Hayut]{School of Mathematical Sciences.
Tel Aviv University.
Tel Aviv 69978,
Israel}
\author{Spencer Unger}
\email[Spencer Unger]{spencer.unger@mail.tau.ac.il}
\address[Spencer Unger]{School of Mathematical Sciences.
Tel Aviv University.
Tel Aviv 69978,
Israel}
\begin{document}
\begin{abstract}
We improve the upper bound for the consistency strength of stationary reflection at successors of singular cardinals.
\end{abstract}
\maketitle
\section{Introduction}

Stationary reflection is an important notion in the investigation of
compactness principles in set theory.  Its failure, the existence of a
nonreflecting stationary set, is sufficient for the construction of objects which
witness the noncompactness of various properties.  Examples include freeness of
abelian groups and metrizability of topological spaces \cite{MagidorShelah1994}
and chromatic number of graphs \cite{Todorcevic1983,Shelah2013}.

We recall the basic definitions:

\begin{definition}
Let $\kappa$ be a regular cardinal. A set $S \subseteq \kappa$ \emph{reflects at $\alpha$} if $S \cap \alpha$ is stationary at $\alpha$, where $\cf \alpha > \omega$.  We say that a stationary set $S \subseteq \kappa$ \emph{reflects} if it reflects at $\alpha$ for some $\alpha < \kappa$.
\end{definition}

\begin{definition}
For a stationary set $S\subseteq \kappa$, we denote by $\Refl(S)$ the assertion: $\forall T \subseteq S$ stationary, $T$ reflects.
\end{definition}

The main theorem of this paper deals with the consistency strength of stationary
reflection at $\aleph_{\omega + 1}$.  Until our work the best known upper bound
is due to Magidor \cite{Magidor1982}.

\begin{theorem}[Magidor]
$\Refl(\aleph_{\omega+1})$ is consistent relative to the existence of $\omega$-many supercompact cardinals.
\end{theorem}

We prove:
\begin{theorem}
$\Refl(\aleph_{\omega+1})$ is consistent relative to the existence of a cardinal
$\kappa$ which is $\kappa^+$-$\Pi^1_1$-subcompact. \end{theorem}

Subcompact cardinals were defined by Jensen, and $\kappa^+$-$\Pi^1_1$-subcompact
cardinals were defined by Neeman and Steel (denoted $\Pi^2_1$-subcompact in
\cite{NeemanSteelSubcompact}).  Under $\GCH$, the large cardinal assumption in
our theorem is weaker than the assumption that $\kappa$ is
$\kappa^+$-supercompact. Subcompactness and its variations were defined during
the investigation of square principles in core models. See Section \ref{section:
subcompact cardinals} for the exact definitions, and more details.

Our construction is motivated by an analogy with the consistency of stationary
reflection at $\aleph_2$. Reflection of stationary sets is an instance of
reflection of a $\Pi^1_1$-statement. Namely, if $S$ is a subset of $\kappa$,
then the statement ``$S$ is stationary'' is a $\Pi^1_1$-statement in the model
$\langle H(\kappa), \in, S\rangle$.  If $\kappa$ is weakly compact, then this
$\Pi^1_1$-statement will reflect to a smaller ordinal $\alpha$. So $S \cap
\alpha$ is stationary and $S$ reflects.

Thus, it was natural to suspect that the consistency strength of
$\Refl(S^{\omega_2}_{\omega})$ is weakly compact. Baumgartner \cite{Baumgartner} showed that after collapsing a weakly compact to be
$\aleph_2$, $\Refl(S^{\omega_2}_{\omega})$ holds and even any collection of
$\aleph_1$ stationary subsets of $S^{\omega_2}_{\omega}$ will reflect at a
common point. This thesis was supported by a result of Jensen that stationary
reflection is possible in $L$ only at weakly compact cardinals. Moreover,
Magidor \cite{Magidor1982} showed that if any two stationary subsets of
$S^{\omega_2}_{\omega}$ have a common reflection point then $\omega_2$ is weakly
compact in $L$.

Surprisingly, in \cite{HarringtonShelah1985}, Shelah and Harrington proved that
the consistency strength of $\Refl(S^{\omega_2}_{\omega})$ is only a Mahlo
cardinal.  An important part of their result is the idea that one must iterate
to destroy the stationarity of certain ``bad" sets to achieve stationary
reflection.  These results show that there is a gap in the consistency strength
between stationary reflection and simultaneous reflection for collections of
stationary sets. This gap is explained by the difference between Jensen's square
$\square_\kappa$ and Todor\v{c}evi\'{c}'s square $\square(\kappa^{+})$. See
\cite{HayutLambieHanson2016} for more details.

In our work, we exploit the strong analogy between weak compactness and
$\Pi_1^1$-subcompactness in order to get the consistency of stationary reflection
at $\aleph_{\omega+1}$.  Our argument is similar to Baumgartner's in the sense
that we do not need to iterate to destroy bad stationary sets.  This analogy
suggests that our assumption is not quite optimal.

There is a vast gap between the strength in the large cardinal axioms which are
needed for stationary reflection at $\aleph_2$ and at $\aleph_{\omega+1}$. This gap is
related to the problem of controlling the successor of a singular cardinal.  The
Weak Covering Lemma \cite{JensenSteel2013} states that if there is no transitive
model with a Woodin cardinal then there is a definable class $K$ which is
generically absolute and for every $\kappa$ which is a strong limit singular
cardinal in $V$, $(\kappa^{+})^V = (\kappa^{+})^K$. In $K$, $\square_\kappa$
holds for all infinite $\kappa$ by a result of Schimmerling and Zeman
\cite{SchimmerlingZeman2001}. Since $\square_\kappa$ is upwards absolute between
models that agree on $\kappa^{+}$, we conclude that if there is no inner model
with a Woodin cardinal, then $\square_\kappa$ holds at every successor of a
singular cardinal and therefore stationary reflection fails by a standard
argument.

Thus, in order to obtain stationary reflection principles at a successor of a
singular cardinal, one needs either to violate weak covering or to start with a
model in which square principles fail, and this requires large cardinal axioms
which are much stronger than the ones which are required for the treatment of successor of regular cardinals.

The paper is organized as follows. In Section \ref{section: Prikry} we prove
some standard facts about Prikry forcing with collapses.  In Section
\ref{section: subcompact cardinals}, we give the definitions of subcompact and
$\Pi_1^1$-subcompact and calibrate the extent to which they imply stationary
reflection.  In Section \ref{section: stationary reflection}, we prove our main
theorem.

\section{Prikry forcing} \label{section: Prikry}
In this section we will review some facts about Prikry forcing which are useful
in this paper. We refer the reader to \cite{Gitik2010} for the proofs of the
facts cited in this section.  For this section we assume that $\kappa$ is a
measurable cardinal and $2^\kappa=\kappa^+$.  Let $\mathcal{U}$ be a normal
ultrafilter on $\kappa$ and $j:V\to M$ be the ultrapower embedding.

\begin{fact}  We have the following:
\begin{enumerate}
\item If $a = \langle \alpha_\xi \mid \xi < \kappa\rangle \in V \cap \,^\kappa M$ then $a\in M$.
\item $|\power^M(j(\kappa))|^V = |j(\kappa^{+})|^V = \kappa^{+}$.
\item\label{item: existence of M-generic} If $\mathbb{P}\in M$ such that \[M\models \mathbb{P}\text{ is }j(\kappa^{+})\text{-cc, }\kappa^{+}\text{-closed forcing notion, } |\mathbb{P}|\leq j(\kappa^{+})\]
then there is $K \in V$ which is an $M$-generic filter for $\mathbb{P}$.
\end{enumerate}
\end{fact}

Using part \ref{item: existence of M-generic}, let $K\subseteq \Col^M(\kappa^{+}, <j(\kappa))$ be an
$M$-generic filter. We define a forcing $\mathbb{P}$ called Prikry forcing over
the measure $\mathcal{U}$ with interleaved collapses using the guiding generic
$K$.

\begin{definition}
Let $\mathbb{P}$ be the following forcing notion with
\[p = \langle c_{-1}, \rho_0, c_0, \rho_1, c_1, \dots, \rho_{n-1}, c_{n-1}, A, C\rangle \in \mathbb{P}\]
if and only if
\begin{enumerate}
\item $0 \leq n < \omega$. $n$ is called the \emph{length} of $p$, and we write $\len p = n$.
\item $\rho_0 < \rho_1 < \dots < \rho_{n-1} < \kappa$ are called the \emph{Prikry points} of the condition $p$.
\item For $0 \leq i \leq n$, $c_{i-1} \in \Col(\rho_{i-1}^{+}, <\rho_i)$ where
for notational convenience we set $\rho_{-1}=\omega$ and (temporarily)
$\rho_n=\kappa$.
\item $A \in \mathcal{U}$, $\min A > \rho_{n-1}, \sup \dom c_{n-1}$.
\item $C$ is a function with domain $A$, for all $\alpha \in A$ $C(\alpha)\in \Col(\alpha^{+}, < \kappa)$, and $[C]_{\mathcal{U}} \in K$. 
\end{enumerate}
For a condition $p$ as above we write $\rho_m^p$, $A^p$, $c_i^p$ and $C^p$
with the obvious meaning.

We define two orderings. The \emph{direct extension} $\leq^{\star}$, is defined as follows.
$p \leq^\star q$ if $\len p = \len q$, $c_i^p \supseteq c^i_q$ for $i \in \{-1, 0, \dots, n-1\}$, $C^p$ is stronger than $C^q$ pointwise and
$A^p \subseteq A^q$. For a condition $p$ of length $n$ and $\rho \in A^p$, 
we denote by $p \frown \rho$ the
condition of length $n+1$ with $\rho_i = \rho_i^p$ for $i < n$, $\rho_n = \rho$, $c_i = c_i^p$ for $i < n$, $c_n = C(\rho)$, 
measure one set $A^p \setminus \sup\dom c_n$ and the natural restriction of $C$. 
The forcing ordering $\leq$ is given by a combination of direct
extensions and adjoining points as above. Namely, $\leq$ is the transitive closure
of the relation \[\{ (p,q) \in \mathbb{P}^2 \mid p \leq^\star q \text{ or }\exists \rho \in A^p, q = p\frown \rho\}.\] 
\end{definition}

For a condition $p\in\mathbb{P}$ as above, the \emph{stem} of $p$, is $\langle
c_{-1}, \rho_0, c_0, \dots, \rho_{n-1}, c_{n-1}\rangle$. Clearly, if $p, p'\in
\mathbb{P}$ have the same stem then they are compatible. In particular,
$\mathbb{P}$ is $\kappa$-centered.  We also note that $\leq^{\star}$ is only $\sigma$-closed.

\begin{lemma}\label{lemma: prikry property}
$\mathbb{P}$ satisfies the Prikry Property. Namely, for every statement in the forcing language $\Phi$ and condition $p\in \mathbb{P}$ there is $q \leq^\star p$ such that $q \Vdash \Phi$ or $q \Vdash \neg \Phi$.
\end{lemma}

Using the Prikry Property and a standard factorization argument, one can show that the set of cardinals below $\kappa$ in the generic extension is exactly $\{\omega, \omega_1\} \cup \{\rho_n, \rho_n^{+} \mid n < \omega\}$. In particular, $\kappa$ is forced to be $\aleph_\omega$ of the generic extension.

Let $p$ be a condition with stem $s$. The set of stems of conditions $q\leq p$, $\len q = \len p$ is naturally isomorphic to a finite product of Levy collapses. We will say that a set of stems $D$ is dense if it is dense with respect to this order.

Lemma \ref{lemma: prikry property} has several stronger versions which are called the Strong Prikry Property. The version which we need is the following:

\begin{lemma}\label{lemma: strong prikry property}
Let $D\subseteq \mathbb{P}$ be a dense open set. 
There are a large set $A\in \mathcal{U}$ and a condition $[C] \in K$ 
such that the following holds. For every condition of the form $p = \stem (p)^\smallfrown \langle A, C\rangle$, 
there is a dense set of extensions for the stem (with the same length), $E$, such that for every $q \leq^\star p$ with 
$\stem q\in E$, there is a natural number $m$ such that for every $q'\leq q$, with $\len q' = \len q + m$, $q' \in D$.
\end{lemma}

Let $p$ be a condition of length $n$,
\[p = \langle c_{-1}, \rho_0, c_0, \rho_1, c_1, \dots, \rho_{n-1}, c_{n-1}, A, C\rangle.\]

Let $\mathbb{P}\restriction p$ be the set of conditions in $\mathbb{P}$ which are stronger than $p$. Let $\mathbb{P}_n\restriction p$ be the forcing 
\[\Col(\omega_1, {<}\rho_0) \times \Col(\rho_0^{+}, {<}\rho_1) \times \cdots \times \Col(\rho_{n-1}^{+}, {<}\kappa)\]
below the condition $(c_{-1}, \dots, c_{n-1})$. Note that we include in the
definition of $\mathbb{P}_n$ the last collapse of all cardinals between
$\rho_{n-1}^+$ and $\kappa$. This will be useful later.

Let $W$ be a model of set theory, $V \subseteq W$. In $W$, let $\langle \rho_n \mid n <
\omega\rangle\in W$ be a sequence of $V$-regular cardinals below $\kappa$ and
let $C_n \subseteq \Col^V(\rho_n^{+}, <\rho_{n+1})$, $C_{-1} \subseteq
\Col^V(\omega_1, <\rho_0)$ be filters.

Let $\mathcal{C} = \langle C_n \mid -1 \leq n < \omega\rangle$, $P = \langle \rho_n \mid n < \omega\rangle$. Let us denote by $G(\mathcal{C}, P)$ the filter which is defined from $\mathcal{C}$ and $P$. Namely,  $G(\mathcal{C}, P) \subseteq \mathbb{P}$ is defined by:
\[p = \langle c_{-1}, \eta_0, \dots, c_{n-2}, \eta_{n-1}, c_{n-1}, A, F\rangle \in G(\mathcal{C}, P)\]
if and only if
\begin{enumerate}
\item $p\in\mathbb{P}$.
\item For all $m \in \omega$ with $m < n$, $\rho_m = \eta_m$, $c_m \in C_m$ (in particular, the domain of $c_{n-1}$ is a subset of $\rho_{n-1}^{+}\times\rho_{n}$). 
\item $c_{-1}\in C_{-1}$.
\item For all $m \geq n$, $\rho_m \in A$ and $F(\rho_m) \in C_m$.
\end{enumerate} 

\begin{theorem}\label{theorem: prikry generic from collapses and sequence}
Let $\mathcal{C}, P\in W$ be as above. $G(\mathcal{C}, P)$ is $V$-generic if and
only if
\begin{enumerate}
\item For every $m \in \{-1\} \cup \omega$, $C_m$ is $V$-generic.
\item For every $A\in \mathcal{U}$, there is $n < \omega$ such that for all $m \geq n$, $\rho_m \in A$.
\item For every $C\colon \kappa \to V$ such that $[C]\in K$ there is $n < \omega$ such that for all $m \geq n$, $C(\rho_m) \in C_m$.
\end{enumerate}
\end{theorem}
\begin{proof}
The forward direction is clear. 

For the backwards direction, let $G$ be the filter generated by $\mathcal{C}, P$. 
Let $D$ be a dense open subset of $\mathbb{P}$.
We will find a condition $r\in G \cap D$.

By the Strong Prikry Property (Lemma \ref{lemma: strong prikry property}), there
are a large set $A$ and a member $[C]$ of $K$ such that for every condition $q$
of the form $\stem(q)^\smallfrown \langle B, F\rangle$,
with $B \subseteq A$ and $\forall \alpha \in \dom(F),\, F(\alpha) \leq
C(\alpha)$, there is dense subset $E$ of the stems of $\mathbb{P}$ below
$\stem(q)$ as in the conclusion of Lemma \ref{lemma: strong prikry property}.
 
Let $q\in G$ of some length $n$ such that for all $m \geq n$, $\rho_m \in A$ and
$C(\rho_m) \in C_m$.  Let $E$ be the witnessing dense open set of stems as
above.  By a standard argument using Easton's lemma, $C_{-1} \times C_0 \times
\cdots \times C_{n-1}$ is $V$-generic.  For notational convenience we call this
generic $C_n^*$.  Since $(c^{q}_{-1}, c_0^{q}, \dots c_{n-1}^{q}) \in C_n^*$,
there is some extension $(c_{-1},c_0, \dots c_{n-1})$ of it in $C_n^* \cap E$.
Let $q'$ be the strengthening of $q$ by $(c_{-1},c_0, \dots c_{n-1})$.

By the conclusion of Lemma \ref{lemma: strong prikry property}, there is a
natural number $m$ such that any $m$-step extension of $q'$ is in $D$.  So if we
take a condition $r \in G$ with $r \leq q'$ of length $n+m$ then it follows that
$r \in D$.\end{proof}

Let $M_0 = V$ and let $j_{n,n} = id$ for all $n < \omega$. Let us define, by induction on $n$, transitive classes $M_n$ and elementary embeddings $j_{m, n} \colon M_m \to M_n$. Let us denote $j_{n} = j_{0, n}$. Let $j_{n, n + 1}\colon M_n \to M_{n+1}$, be the ultrapower by $j_{0, n}(\mathcal{U})$ and let $j_{m, n+1} = j_{n, n+1} \circ j_{m, n}$ for every $m < n$. 

Let $j_{\omega}\colon V \to M_{\omega}$ be the direct limit of the directed system $\langle M_m, j_{m,n} \mid m \leq n < \omega\rangle$. Let $j_{n, \omega}\colon M_n\to M_{\omega}$ be the corresponding elementary embeddings.
\begin{theorem}[Gaifman]
$M_\omega$ is well founded.
\end{theorem}
The following fact is well-known:
\begin{lemma}\label{lemma: MomegaP is closed}
$M_{\omega}[\langle j_{n}(\kappa) \mid n < \omega\rangle]$ is closed under $\kappa$-sequences.
\end{lemma}
\begin{proof}
Since $W = M_{\omega}[\langle j_{n}(\kappa) \mid n < \omega\rangle]$ is a model of $\ZFC$, it is enough to show that for every $\kappa$-sequence of ordinals from $V$, $s=\langle \alpha_\xi \mid \xi < \kappa\rangle$, belongs to $W$. 

Let us fix for every $\xi < \kappa$, a natural number $n_\xi$ and a function $f\colon \kappa^{n_{\xi}} \to \Ord$ such that $j_{\omega}(f)(\kappa, j_1(\kappa), \dots, j_{n_\xi-1}(\kappa)) = \alpha_\xi$. 

Let $F = \langle f_\xi \mid \xi < \kappa\rangle$. $j_{\omega}(F) \restriction
\kappa = \langle j_{\omega}(f_\xi) \mid \xi < \kappa\rangle$. By applying each
function from the sequence $j_{\omega}(F) \restriction \kappa$ to the corresponding initial segment of $\langle j_n(\kappa) \mid n < \omega\rangle$, we get $s$.  
Since $j_\omega(F)\restriction \kappa \in W$, we conclude that $s\in W$.
\end{proof}

\begin{definition}
For a subset $X$ of a partial order $\mathbb{X}$ we will denote by ${<}X{>}$ the upwards closure of $X$:
\[{<}X{>} = \{x \in \mathbb{X} \mid \exists y\in X,\, x \geq y\}.\]
\end{definition}

The following well-known fact will play a major role in Section \ref{section: stationary reflection}.
\begin{lemma}\label{lemma: prikry generic over Momega}
Let $p \in \mathbb{P}$, $\len p = n$. Let $G' \subseteq \mathbb{P}_n\restriction p$ be a $V$-generic filter. In $V[G']$ there is an $M_\omega$-generic filter for $j_\omega(\mathbb{P})$ that contains $j_\omega(p)$.
\end{lemma}
\begin{proof}
Let $n > 0$. Let $K_n \subseteq \Col(j_{n-1}(\kappa)^{+}, <j_{n}(\kappa))^{M_n}$ be ${<}j_{n-1}\image K{>}$, i.e.\ the filter which is generated from $j_{n-1}\image K$. Note that $K_n = j_{n-1}(K)$. The following argument is standard, see \cite{CummingsWoodin}. 
\begin{claim}
$K_n$ is $M_n$-generic. 
\end{claim}
\begin{proof} 
Let $D \in M_n$ be a dense open set. Then there is a function $f\colon \kappa^n
\to V$ such that for all $a \in \kappa^n$, $f(a)$ is a dense open subset of the
forcing $\Col(a_{n-1}^{+}, <\kappa)$ where we write $a = \{a_0,a_1, \dots
a_{n-1}\}$ listed in an increasing order. 
By the distributivity of the forcing $\Col(a_{n-1}^{+}, <\kappa)$, 
for every $\alpha < \kappa$, the set $D_\alpha = \bigcap_{b \in \alpha^{n-1}} f(b^\smallfrown \langle\alpha\rangle)$ 
is dense open in $\Col(\alpha^{+},<\kappa)$. Let $q$ be a function with domain $\kappa$ such that 
$[q]_{\mathcal{U}} = j_1(q)(\kappa) \in K$, and 
$\{\alpha < \kappa \mid q(\alpha) \in D_\alpha\}\in \mathcal{U}$. 
Such a condition exists, since $K$ is $M_1$-generic. 
Let us consider the function $\tilde{q} \colon \kappa^n \to V$ which is 
defined as $\tilde{q}(a) = q(a_{n-1})$. 

Let $r = j_n(\tilde{q})(\kappa, j_1(\kappa), \dots, j_{n-1}(\kappa))$. $r = j_{n-1}(j_1(q)(\kappa))$ since:

\[\begin{matrix}
j_{n-1}(j_1(q)(\kappa)) & = & j_{n-1}([q]_{\mathcal{U}}) \\ 
& = & [j_{n-1}(q)]_{j_{n-1}(\mathcal{U})} \\ 
& = & j_{n-1, n}(j_{n-1}(q))(j_{n-1}(\kappa)) \\ 
& = & j_n(q)(j_{n-1}(\kappa))
\end{matrix}\]  
and by the definition of $\tilde{q}$, $j_{n}(\tilde{q})(\kappa, \dots, j_{n-1}(\kappa)) = j_n(q)(j_{n-1}(\kappa))$. We conclude that $r \in D$.

Since $r = j_{n-1}([q]_{\mathcal{U}})$, $r\in K_n$. 
\end{proof}
Note that $j_{n, \omega}\image K_n = K_n$ and that $K_n$ is also an $M_\omega$-generic filter.

Let $\mathcal{C}$ be the sequence of generic collapses from $G'$, augmented by the sequence $\langle K_1, K_2, \dots\rangle$. Let $P = \langle \rho_0, \dots, \rho_{n-1}, \kappa, j_1(\kappa), j_2(\kappa),\dots\rangle$, where $\rho_0, \dots, \rho_{n-1}$ are the Prikry points in the condition $p$. Let $G = G(\mathcal{C}, P) \subseteq j_\omega(\mathbb{P})$, as in Theorem \ref{theorem: prikry generic from collapses and sequence}. 

For every $A\in j_{\omega}(\mathcal{U})$ there is $m < \omega$ such that $A =
j_{m, \omega}(A')$. Note that the tail of the sequence $P$, starting at point $n
+ m$ is contained in $A$. Similarly, if $q' \in j_{\omega}(K)$ then $q' =
j_{m,\omega}(q)$ for some $q$ and therefore for every $k \geq m$,
$q(j_{k}(\kappa)) \in K_k$. Finally, each $K_n$ is $M_{\omega}$-generic. Thus,
the conditions of Theorem \ref{theorem: prikry generic from collapses and
sequence} hold and $G$ is $M_\omega$-generic for $j_\omega(\mathbb{P})$.
\end{proof}

\subsection{Splitting generic filters}\label{subsection: curly H}
During the proof of the main theorem, Theorem \ref{theorem: preserving stationary reflection}, we will need to analyze models of the form $M_{\omega}[P][\mathcal{H}]$ such that $P$ is the critical sequence and $\mathcal{H}$ has the form $\langle {<}j_{n,\omega} \image H{>} \mid n < \omega\rangle$ where $H$ is a $V$-generic filter for some $\kappa^{+}$-closed forcing notion in $M_{\omega}[P]$.  

Let $\mathbb{A}$ be a forcing notion that has unique greatest lower bounds and a $\kappa^{+}$-closed dense subset, which is closed under those greatest lower bounds. Let us assume that $j_n(\mathbb{A}) = \mathbb{A}$ for all $n < \omega$ (in particular, $\mathbb{A} \in M_{\omega}[P]$). In this subsection we will define and analyze a forcing notion, $\mathbb{H}$, which will have the property that $P$ and $\mathcal{H}$ generate an $M_{\omega}$-generic filter for $j_{\omega}(\mathbb{H})$. 

A lot of information on the model $M_{\omega}[P][\mathcal{H}]$ can be extracted without understanding the forcing $\mathbb{H}$. In particular, one can prove Lemmas \ref{lemma: distributivity of cal H} and \ref{lemma: Momega H is intersection} without mentioning $\mathbb{H}$. Moreover, by applying general arguments one may apply Bukovski's Theorem, \cite{Bukovsky1973}, and deduce the \emph{existence} of some forcing notion that introduces $\mathcal{H}$ over $M_{\omega}[P]$, without knowing what precisely this forcing is. In particular, one can prove Claim \ref{claim: stationary reflection in Momega}, which is central in the Theorem \ref{theorem: preserving stationary reflection}, without explicitly constructing the forcing notion $\mathbb{H}$. Despite this, we prefer to construct the forcing $\mathbb{H}$ in details, since we believe that its structure helps to unravel some of the mysterious properties of $\mathcal{H}$. 

In the following definition we will use the convention that a finite sequence $s$ is an end extension of a sequence $t$ if $t = s \restriction \len t$. In this case we will write $t \trianglelefteq s$. Note that $s \trianglelefteq s$ always holds. We will denote by $s \perp t$ if $s\not\trianglelefteq t$ and $t\not\trianglelefteq s$.

The conditions of $\mathbb{H}$ are pairs of the form $p = \langle T, F\rangle$ where: 
\begin{enumerate}
\item $T \subseteq \kappa^{{<}\omega}$. For every $\eta\in T$ and $n < \len \eta$, $\eta \restriction n \in T$. Let us order $T$ by $\trianglelefteq$. 
\item Any element of $T$ is a strictly increasing finite sequence of regular cardinals.
\item There is a single element, $s \in T$ such that every $t\in T$ is comparable with $s$ and $\len s$ is maximal. Let us denote $\stem(T) = s$. 
\item For every $t \in T$, if $\stem(T) \trianglelefteq t$ then
\[\{\alpha < \kappa \mid t ^\smallfrown \langle\alpha\rangle\in T\}\in\mathcal{U}.\]
\item $F$ is a function, $F\colon T \to \mathbb{A}$.
\item \label{requirement: stabilization}(Stabilization) Let $t\in T$ such that $\stem(T) \trianglelefteq t$. Let $g_t$ be the function $g_t(\alpha) = F(t^\smallfrown \langle\alpha\rangle)$ for all $\alpha < \kappa$ such that  $t^\smallfrown \langle\alpha\rangle\in T$. Then $j(g_t)(\kappa) = F(t)$. \end{enumerate}
For a condition $p = \langle T, F\rangle\in\mathbb{H}$ we write $T^p = T$, $F^p = F$, and $\len p = \len \stem(T)$. We denote $\stem(p) = \langle F(t) \mid t \trianglelefteq \stem(T)\rangle$.

For $p, q\in \mathbb{H}$, we define $p \leq q$ ($p$ extends $q$) if $T^p \subseteq T^q$ and $F^p(\eta) \leq_{\mathbb{A}} F^q(\eta)$ for all $\eta\in T^p$. We define $p\leq^\star q$ ($p$ is a direct extension of $q$) if $p\leq q$ and $\len p = \len q$.

\begin{lemma}[Strong Prikry Property]
Let $D \subseteq \mathbb{H}$ be a dense open set and let $p\in\mathbb{H}$. There is a direct extension $p^\star \leq p$ and a natural number $n < \omega$ such that any $q \leq p^\star$ with $\len q \geq n$ is in $D$.
\end{lemma}
\begin{proof}
Let $D\subseteq \mathbb{H}$ be dense open and $p\in \mathbb{H}$ be a condition.

Let $\langle \eta_\alpha \mid \alpha < \kappa\rangle$ be an enumeration of
$\kappa^{{<}\omega}$ such that if $\eta_\alpha \trianglelefteq \eta_\beta$, then
$\alpha \leq \beta$.  Let us define, by recursion, a decreasing sequence of
conditions $p_\alpha = \langle T_\alpha, F_\alpha\rangle$, $\alpha < \kappa$, such that $\stem(T_\alpha) = \stem(T_\beta)$ for all $\alpha < \beta$. For all such conditions, the range of $F_\alpha$ is always chosen to be included in the $\kappa^{+}$-closed subset of $\mathbb{A}$.

Let $p = p_0$. Let $s = \stem(T_0)$. For each $\alpha$, if $\eta_\alpha\in T_\alpha$ end extends
$s$, we look at the tree $T_{\eta_\alpha} = \{\eta \in T_\alpha \mid \eta_\alpha \not\perp \eta\}$ and
check if there is a condition $q_\alpha= \langle T', F'\rangle\in D$, which is a
direct extension of $p_\alpha\restriction T_{\eta_\alpha}$. If there is no such condition, we let $p_{\alpha + 1} = p_\alpha$.

Otherwise, let us define: 
\[T'' = \{\eta\in T_\alpha \mid \eta \perp \eta_\alpha \text{ or } \eta \in T'\},\] 
\[F'' = F_{\alpha} \restriction \{\eta\in T_\alpha \mid \eta \perp \eta_\alpha\} \cup F'.\]

For all $\eta \in T'' \setminus T'$, requirement \ref{requirement: stabilization} might fail. Since $F''(\eta)$ is stronger than $F_\alpha(\eta)$, we may find a condition $r_{\eta^\smallfrown \langle\alpha\rangle}$ such that $r_{\eta^\smallfrown \langle\alpha\rangle} \leq F''(\eta^\smallfrown\langle\alpha\rangle)$ for all $\alpha$ and $j(r)_{\eta^\smallfrown\langle \kappa\rangle} = F''(\eta)$. Continue this way for $\omega$ many steps we construct a function $F'''$ with domain $T''$ such that $p_{\alpha + 1} = \langle T'', F'''\rangle$ is a condition.

For a limit ordinal $\alpha < \kappa$, let $p_\alpha = \langle T_\alpha, F_\alpha\rangle$ be the pair, $T_\alpha = \bigcap_{\beta < \alpha} T_{\beta}$ and $F_\alpha(\eta)$ is the greatest lower bound of $F_\beta(\eta)$ for all $\eta\in T_\alpha$ (this lower bound exists by the closure of the forcing $\mathbb{A}$). 

Let us verify that for all $\alpha < \kappa$, $p_\alpha$ is a condition. For limit ordinal $\alpha$, $T_\alpha$ is $\mathcal{U}$-splitting, using the closure of the measure $\mathcal{U}$. $F_\alpha$ satisfies condition \ref{requirement: stabilization}, since for all $\eta\in T_\alpha$, $F_\alpha(\eta)$ is the greatest lower bound of a decreasing sequence of length $\alpha < \kappa$. Since $\alpha < \crit j$, applying $j$ does not change this fact. For successor ordinals $\alpha$, the requirement follows from the construction.

We would like to continue and construct a condition $p_\kappa$, which is a lower bound for the sequence $p_\alpha$. 

Let $T_\kappa = \bigcap_{\alpha < \kappa} T_\alpha$. Since the set of successors of each element in the tree is modified only finitely many times, $T_\kappa$ is a $\mathcal{U}$-splitting tree with stem $s$.

Let us consider for each $\eta \in T_\kappa$ the following sequence of functions. Let $g^n_{\eta}\colon \kappa^n\to\mathbb{A}$ be the function defined by $g^n_\eta(\nu) = \lim_{\alpha} F_{\alpha}(\eta^\smallfrown \nu)$. Let 
\[p^n_\eta = j_n(g^n_\eta)(\kappa, j_1(\kappa), \dots, j_{n-1}(\kappa)).\]
Let us claim that for each $\eta$, the sequence $\langle p^n_\eta \mid n < \omega\rangle$ is decreasing. Indeed, 
\[\begin{matrix}
p^1_\eta & = & j_1(g^1_\eta)(\kappa) & = & \lim_{\alpha < j_1(\kappa)} j_1(F)_\alpha(\eta^\smallfrown \langle \kappa\rangle) \\ 
& \leq &\lim_{\alpha < \kappa} j_1(F)_\alpha(\eta^\smallfrown \langle \kappa\rangle) & = & \lim_{\alpha < \kappa} j_1(F)_\alpha(\eta^\smallfrown \langle \kappa\rangle) \\ 
& = & \lim_{\alpha < \kappa} F_\alpha(\eta) & = & p^0_\eta
\end{matrix}\]

By the elementarity of $j_n$, 
\[j_n(\langle F_\alpha \mid \alpha < \kappa\rangle) = \langle F^{j_n}_\alpha \mid \alpha < j_n(\kappa)\rangle\]
and for all $\alpha < j_n(\kappa)$, for all $\eta \in \dom F^{j_n}_\alpha$, \[j_{n,n+1}(\langle F^{j_n}(\eta \smallfrown \langle \zeta\rangle) \mid \zeta < j_n(\kappa)\rangle)(j_n(\kappa)) = F^{j_n}(\eta).\] 

and therefore, we conclude that in general:
\[\begin{matrix}
p^{n + 1}_\eta & = & j_{n + 1}(g^{n + 1}_\eta)(\kappa, \dots, j_n(\kappa)) \\ 
& = & \lim_{\alpha < j_{n + 1}(\kappa)}\ F^{j_{n+1}}_\alpha(\eta^\smallfrown \langle \kappa, \dots, j_n(\kappa)\rangle) \\ 
& \leq &\lim_{\alpha < j_{n}(\kappa)}\ j_{n+1}(F)_\alpha(\eta^\smallfrown \langle\kappa ,\dots, j_n(\kappa) \rangle) \\ 
& = & \lim_{\alpha < j_n(\kappa)}\ j_{n,n+1}(F^{j_n}_\alpha)(\eta^\smallfrown \langle \kappa, \dots, j_{n-1}(\kappa)\rangle ^\smallfrown \langle j_n(\kappa)\rangle) \\ 
& = & \lim_{\alpha < j_n(\kappa)} (F^{j_n}_\alpha)(\eta^\smallfrown \langle \kappa, \dots, j_{n-1}(\kappa)\rangle) \\
& = &  j_n(g^n_\eta)(\eta^\smallfrown \langle \kappa, \dots, j_{n-1}(\kappa)\rangle) & = & p^n_\eta.
\end{matrix}\]

Let $F(\eta)$ be the greatest lower bound of $\langle p^n_\eta \mid n < \omega\rangle$. Let us claim that requirement \ref{requirement: stabilization} holds for $F$. Namely, that for any $\eta$, if we let $h_\eta(\alpha) = F(\eta^\smallfrown\langle\alpha\rangle)$, then $j(h_\eta)(\kappa) = F(\eta)$. By elementarity, $j(h_\eta)(\kappa)$ is the greatest lower bound of the sequence $q^n_\eta$ where $q^n_\eta = j(\langle p^n_{\eta^\smallfrown\langle\alpha\rangle} \mid \alpha < \kappa\rangle)(\kappa)$. Let us compute:
\[\begin{matrix}
q^n_\eta & = & j(\langle p^n_{\eta^\smallfrown\langle\alpha\rangle} \mid \alpha < \kappa\rangle)(\kappa) & & \\ 
& = & j_1(\langle j_n(g^n_{\eta^\smallfrown\langle\alpha\rangle})(\kappa, \dots, j_{n-1}(\kappa)) \mid \alpha < \kappa\rangle)(\kappa) & &\\
& = & j_{n + 1}(g^n)_{\eta ^\smallfrown\langle\kappa\rangle}(j_1(\kappa), \dots, j_{n}(\kappa)) \\

& = & j_{n+1}(g^{n+1}_\eta)(\kappa, j_1(\kappa), \dots, j_n(\kappa)) & = & p^{n+1}_\eta(\kappa). 
\end{matrix}\] 

Thus, $F(\eta)$ satisfies requirement \ref{requirement: stabilization} in the definition of $\mathbb{H}$. Let $p_\kappa = \langle T_\kappa, F\rangle$.

Let us consider the condition $p_\kappa$. By narrowing down the tree $T = T_\kappa$, we may assume that for any $\eta\in T$ one of the following two holds: Either for every $\alpha < \kappa$, if $\eta^\smallfrown \langle\alpha\rangle\in T$ then the extension of $p$ by picking $\eta^\smallfrown\langle\alpha\rangle$ is in $D$, or that all of them are not in $D$. By narrowing the tree $T$ again, we may assume that for any element of the tree $\eta$ the minimal level of the tree that enters $D$ above $\eta$ is fixed. This induces a coloring of $T$ which is (by induction) fixed on levels. Clearly, if an element $\eta\in T$ was colored by the number $n$ then its successors are colored by $n-1$. Let $n$ be the color of the root of $T$, $\emptyset$. Then, every direct extension of $p_\kappa$ with length $\geq n$ is in $D$, as required.
\end{proof}
The generic filter can be described compactly using a branch in the tree, $P$, and a sequence of filters $\mathcal{H} = \langle H_n \mid n < \omega\rangle$ of $\mathbb{A}$. Let $G(P,\mathcal{H})$ be the set of all conditions $p\in\mathbb{H}$ such that for all $n < \omega$, $P\restriction n \in T^p$ and $F^p(P\restriction n) \in H_n$.
\begin{lemma}\label{lemma: G P calH is filter}
For any increasing $\omega$-sequence in $\kappa$ and collection of filters $\mathcal{H}=\langle H_n \mid n <\omega\rangle$, $G(P,\mathcal{H})$ is a filter.
\end{lemma}
\begin{proof}
Let $p, q \in G(P, \mathcal{H})$. We want to show that they are compatible and have a common lower bound in $G(P, \mathcal{H})$. Let us assume, without loss of generality, that $\len p \leq \len q$. Then $\stem(T^p) \trianglelefteq \stem (T^q) \trianglelefteq P$. In particular, the intersection of $T^p$ and $T^q$ is a $\mathcal{U}$-branching tree. Thus we may assume without loss of generality that $T^p = T^q = T$.

Let us consider $F^p$ and $F^q$. For every element below the stem of the tree, the values of those two functions are compatible since each $H_m$ is a filter. For elements above the stem, one can show by induction on the height of $\eta \in T$, that the compatibility of $F^p(\eta)$ and $F^q(\eta)$ implies (using requirement \ref{requirement: stabilization} in the definition of the forcing) that for a large set of extension of $\eta$, $F^q(\eta^\smallfrown\langle\alpha\rangle)$ is compatible with $F^p(\eta^\smallfrown\langle\alpha\rangle)$. Moreover, by the definition of the filter $G(P,\mathcal{H})$, if $\eta$ is an initial segment of $P$ then $F^p(\eta)$ is compatible with $F^q(\eta)$. Narrowing down $T$, we may assume that for any $\eta\in T$, $F^p(\eta)$ is compatible with $F^q(\eta)$, while the initial segments of $P$ are all in $T$. Let $F(\eta)$ be the greatest lower bound of $F^p(\eta)$ and $F^q(\eta)$.  For every $\eta$, let $g_\eta(\alpha) = F(\eta^\smallfrown\langle\alpha\rangle)$ and let $g^p, g^q$ be the analogous functions with respect to $p$ and $q$. Then, $j(g_\eta)(\kappa)$ is the greatest lower bound of $j(g_\eta^p)(\kappa)$ and $j(g_\eta^q)(\kappa)$ and therefore is the same as $F(\eta)$.

We conclude that this condition satisfies requirement \ref{requirement: stabilization}. Thus, it is in $G(P,\mathcal{H})$, as wanted.
\end{proof}

We wish to generate an $M_\omega$-generic filter for $j_\omega(\mathbb{H})$.
Let $H \subseteq \mathbb{A}$ be a $V$-generic filter.  Let $\mathcal{H}_n =
\langle {<}j_{m,n}\image H {>} \mid m \leq n\rangle$ and let $\mathcal{H} =
\langle {<}j_{n,\omega}\image H{>} \mid n < \omega\rangle$. Let $P$ be the
critical sequence $\langle j_n(\kappa) \mid n < \omega\rangle$.  It is immediate
that $P,\mathcal{H} \in M_n[\mathcal{H}_n]$.

We will show that $G(P,\mathcal{H})$ is $M_\omega$-generic for
$j_\omega(\mathbb{H})$.  We start by showing that $\mathcal{H}_n$ is generic
over $M_n$.  To do so we need the following lemma which is attributed to Woodin
in a paper of Cummings, see \cite[Fact 2]{CummingsWoodin}.

\begin{lemma}[Cummings-Woodin]\label{lemma: cummings-woodin}
Let $W$ be a model of $\ZFC$. In $W$, let $\mu$ be a measurable cardinal, and let $\mathcal{U}$ be a normal measure on $\mu$. Let $j\colon W \to N$ be the ultrapower embedding. 

Let $\mathbb{A}$ be a forcing notion such that $\mathbb{A} = j(\mathbb{B})$ and $\mathbb{B}$ is a $\mu$-closed forcing notion. Let $G_A \subseteq \mathbb{A}$ be a $W$-generic filter. Then $j$ extends to an embedding: 
\[j^\star \colon W[G_A] \to N[{<}j\image G_A{>}]\]
which is definable in $W[G_A]$ 
and $G_A$ is $N[{<}j\image G_A{>}]$-generic. 
\end{lemma}


\begin{claim}\label{claim: Hn is generic} $\mathcal{H}_n$
generates a generic filter for $\mathbb{A}^{n+1}$ over $M_n$.  \end{claim}

\begin{proof} We go by induction on $n<\omega$.  For $n = 0$, this is true as $H$ is a $V$-generic for $\mathbb{A}$.  Assume that the claim holds for $n < \omega$. 

Consider the elementary embedding:
\[j_{n,n+1}\colon M_n \to M_{n+1}\] 
This is an ultrapower embedding, using the measure $j_n(\mathcal{U})$ over
$M_n$.  We apply Lemma \ref{lemma: cummings-woodin} with $W = M_n$, $\mu = j_n(\kappa)$, $j = j_{n,n+1}$, $\mathbb{B} = \mathbb{A}^{n+1}$, and $G_A = \mathcal{H}_n$. We conclude that there is an elementary embedding:
\[j_{n,n+1}^\star\colon M_n[\mathcal{H}_n] \to M_{n+1}[{<}j_{n,n+1}\image \mathcal{H}_n{>}].\] 
By the definition of $\mathcal{H}_n$, ${<}j_{n,n+1}\image \mathcal{H}_n{>} = \mathcal{H}_{n+1}\restriction [1, n+1]$: the last $n+1$ coordinates of $\mathcal{H}_{n+1}$.

By the second part of the lemma, $\mathcal{H}_n$ is $M_{n+1}[\mathcal{H}_{n+1}\restriction [1, n+1]]$-generic. In particular, $H$ is generic over this model. Since $H$ is the first component of $\mathcal{H}_{n+1}$ we conclude that $\mathcal{H}_{n+1}$ is $M_{n+1}$-generic for the forcing $\mathbb{A}^{n+2}$. \end{proof}

\begin{theorem}\label{theorem: cal H is generic}
$G(P,\mathcal{H})$ is $M_{\omega}$-generic for $j_{\omega}(\mathbb{H})$.
\end{theorem}
\begin{proof}
Let $D\in M_{\omega}$ be dense open and let $p\in G(P,\mathcal{H})$. Let $n < \omega$ be large enough so that there is $D' \subseteq j_n(\mathbb{H})$ and $p' \in j_n(\mathbb{H})$ such that $D = j_{n,\omega}(D')$, $p = j_{n,\omega}(p')$. Without loss of generality, $\len p = n$.

Let
\[E = \{p^\star \in j_n(\mathbb{H}) \mid p^\star \leq^\star p',\, \exists m < \omega, \forall r\leq p^\star, \len r \geq m\implies r\in D'\}.\]
By the Strong Prikry Property in $M_n$, $E$ is dense and open in $\leq^\star$.

Let us claim that there is a condition $r\in E$ such that 
$j_{n,\omega}(r) \in G(P,\mathcal{H})$. Indeed, since $E$  is dense open in the
direct extension relation, the collection of stems of elements of $E$ is dense
open in the forcing $\mathbb{A}^{n+1}$. By the genericity of $\mathcal{H}_n$,
there is $s\in \mathcal{H}_n$ which is a stem of an element in $E$. Let $r$ be
any element in $E$ with $\stem(r) = s$. Let us claim that $j_{n,\omega}(r)\in
G(P,\mathcal{H})$. Indeed, working in $V$ fix a condition $h\in\mathbb{H}$ such
that $s' = \langle j_{m,n}(h) \mid m \leq n\rangle$ is stronger than $s$. Let
$r' = \langle T', F'\rangle$ be any condition of $\mathbb{H}$ such that
$\stem(r') = s'$. Then, by induction on $k < \omega$, \[j_{n, n +
k}(F')(\stem(T') ^\smallfrown \langle j_n(\kappa), \dots, j_{n + k -
1}(\kappa)\rangle) = h.\] We conclude that $j_{n,\omega}(F')(\langle \kappa,
\dots, j_{k-1}(\kappa)\rangle) \in j_{k,\omega}\image H$ for all $k$. Thus,
$j_{n,\omega}(r')\in G(P,\mathcal{H})$. By the arguments of Lemma \ref{lemma: G P calH is filter},  
by narrowing down the tree of $r'$ in $M_n$, we obtain a condition $r'' \leq^\star r'$, $F^{r''} = F^{r'} \restriction T^{r''}$ and $r'' \leq r$. Since the critical sequence starting from $n$ enters any $\mathcal{U}$-splitting tree from $M_n$, $j_{n,\omega}(r'')\in G(P,\mathcal{H})$ as well. 

We conclude that $j_{n,\omega}(r)\in G(P,\mathcal{H})$. Let $m< \omega$ be the natural number that witnesses $r\in E$. I.e.\ for all $q \leq r$ with $\len q \geq m$, $q\in D'$. By elementarity, any extension of $j_{n,\omega}(r)$ of length $m$ is in $D$. In particular, $G(P,\mathcal{H}) \cap D \neq \emptyset$.
\end{proof}

\begin{lemma}\label{lemma: distributivity of cal H}
$M_{\omega}[P][\mathcal{H}]$ has the same $j_\omega(\kappa)$-sequences of ordinals as $M_{\omega}[P]$.
\end{lemma}
\begin{proof}
Since $j_{\omega}(\kappa)$ is singular in $M_{\omega}[P]$, it is enough to show that there is no new sequence of ordinals in $M_{\omega}[P][\mathcal{H}]$ of length $\rho < j_{\omega}(\kappa)$.

Let us fix $n < \omega$ large enough so that $\rho < j_n(\kappa)$. $M_{\omega}[P]$ and $M_n$ have the same $j_n(\kappa)$-sequences of ordinals by Lemma \ref{lemma: MomegaP is closed}. Let us assume that there is a $\rho$-sequence of ordinals, \[x \in M_{\omega}[P][\mathcal{H}] \setminus M_{\omega}[P].\]

Recall that $\mathcal{H}_n = \langle {<}j_{m,n}\image H{>} \mid m \leq
n\rangle$. By Claim \ref{claim: Hn is generic}, $\mathcal{H}_n$ generates a generic filter for
$\mathbb{A}^{n+1}$ which is a $j_n(\kappa^{+})$-closed forcing in $M_n$. In particular, since $\mathcal{H}_n$ is an $M_n$-generic filter for a $j_n(\kappa^{+})$-distributive forcing, $M_n$ has the same $\rho$-sequences as $M_n[\mathcal{H}_n]$ and in particular, any $\rho$-sequence in $M_\omega[P][\mathcal{H}]$ belongs to $M_n$. But since $M_{\omega}[P]$ contains all $j_n(\kappa)$ sequence of ordinals from $M_n$ (by applying Lemma \ref{lemma: MomegaP is closed} in $M_n$), we conclude that $x\in M_{\omega}[P]$.
\end{proof}

Let us give another argument for the distributivity of the extension by $\mathcal{H}$. 
\begin{definition}\label{definition: pure extension of H}
Let $p, q\in \mathbb{H}$. $p \leq^{\star\star} q$ if there is a $\mathcal{U}$-splitting tree $T$ with $\stem(T) = \stem(T^p) = \stem(T^q)$ and $p' = \langle T, F^p\restriction T\rangle \leq^{\star} q$. 
\end{definition}
Using diagonal intersections, the partial order $\leq^{\star\star}$ is $\kappa^{+}$-closed. Let $\langle D_i \mid i < \kappa\rangle$ be a sequence of dense open subsets of $\mathbb{H}$. Using the Strong Prikry Property, we can construct a sequence of conditions $\langle p_i \mid i < \kappa\rangle$ which is decreasing in $\leq^\star$. Let $p_\kappa$ be a $\leq^{\star\star}$-lower bound for the sequence of $p_i$-s. For every $\alpha < \kappa$, there is a natural number $n_\alpha$ such that for every increasing sequence $s\in T^{p_\kappa}$, with $\len s \geq n_\alpha$, $\max s \geq \alpha$, the condition $p' = \langle T^{p_\kappa}_{s}, F^{p_\kappa} \restriction T^{p_\kappa}_s\rangle$ is in $D_\alpha$ (where $T_s$ is the restriction of the tree $T$ to its elements above $s$). This implies that for any name for a $\kappa$-sequence of ordinals there is a condition that forces it to be equivalent to a name relative to the Prikry forcing.

Recall that $\mathbb{P}$ is the Prikry forcing with interleaved collapses and that $G$ is an $M_{\omega}$-generic filter for $j_{\omega}(\mathbb{P})$.

\begin{claim}
In $M_\omega$, $P$ and $\mathcal{H}$ generate a generic filter for
$j_\omega(\mathbb{H})$ which is mutually generic to the quotient forcing for adding collapses over the standard Prikry forcing. Moreover, the filter $\mathcal{H}$ does not add any $j_\omega(\kappa)$-sequences of ordinals over $M_{\omega}[G]$.
\end{claim}

\begin{proof}
Work over $V$. Let $\tilde{P}$ be a $V$-generic Prikry sequence. By elementarity, there is a condition in $\mathbb{H} / \tilde{P}$ that forces that any $\kappa$-sequence of ordinals in the generic extension is already in $V[\tilde{P}]$. The quotient forcing for adding the interleaved collapses over the Prikry sequence, $\mathbb{P} / \tilde{P}$ is $\kappa$-centered and in particular $\kappa^{+}$-cc, also in the extension by $\mathbb{H}$.

Thus, by the distributivity of $\mathbb{H} / \tilde{P}$ over $V[\tilde{P}]$, every maximal antichain of $\mathbb{P} / \tilde{P}$ belongs to $V[\tilde{P}]$. Therefore if $G\subseteq\mathbb{P}$ is a $V$-generic and $\tilde{\mathcal{H}}\subseteq \mathbb{H} / \tilde{P}$ is $V[\tilde{P}]$-generic, then it is also $V[G]$-generic.

The arguments for the distributivity of $j_{\omega}(\mathbb{H}) / P$ over
$M_{\omega}[G]$ are the same as in Lemma \ref{lemma: distributivity of cal H}.
Using the fact that $M_{\omega}[G][\mathcal{H}] \subseteq \bigcap_{n < \omega} M_n[G \restriction n + 1][\mathcal{H}_n]$ and that the forcing that adds $G \restriction n + 1$ is $j_n(\kappa)$-cc, we can trace back any name for a new sequence of ordinals which is shorter than $j_\omega(\kappa)$ to one of the $M_n$ and use the distributivity of $\mathbb{A}^{n+1}$ in $M_n$ in order to conclude that this name appears already in $M_{\omega}[P]$. 
\end{proof}

The following lemma is a generalization of the classical theorem of Bukovsk\'{y} \cite{Bukovsky1977} and independently Dehornoy \cite{Dehornoy1978}. We will follow Bukovsk\'{y}'s proof. 
\begin{lemma}\label{lemma: Momega H is intersection}
$M_{\omega}[P][\mathcal{H}] = \bigcap_{n < \omega} M_n[\mathcal{H}_n]$.
\end{lemma}
\begin{proof}
We already know that $M_{\omega}[P][\mathcal{H}] \subseteq \bigcap_{n < \omega} M_n[\mathcal{H}_n]$. Let us show the other direction.

Let $x$ be a set of ordinals in the intersection of all $M_n[\mathcal{H}_n]$. Since $j_{n,\omega}$ is definable in $M_n$, we may define:
\[x_n = \{\zeta \mid j_{n,\omega}(\zeta) \in x\}\in M_n[\mathcal{H}_n].\]
The embedding $j_{n,\omega}$ extends to an embedding:
\[j_{n,\omega}^\star\colon M_n[\mathcal{H}_n]\to M_{\omega}[<j_{n,\omega}\image \mathcal{H}_n>]\subseteq M_{\omega}[P][\mathcal{H}].\]
Therefore, $j_{n,\omega}(x_n)$ is well-defined and belongs to $M_{\omega}[P][\mathcal{H}]$. The model $M_{\omega}[P]$ is closed under $\omega$-sequences. Since any $H$ does not add any $\omega$-sequence of ordinals to $V$, the same holds in $M_{\omega}[P][\mathcal{H}]$. Thus, $\langle j_{n,\omega}(x_n) \mid n < \omega\rangle\in M_{\omega}[P][\mathcal{H}]$. 

Now we can reconstruct $x$ as follows: $\zeta\in x$ if and only if for all but finitely many $n < \omega$, $\zeta \in j_{n,\omega}(x_n)$.  
\end{proof}
We conclude that $H \in M_\omega[P][\mathcal{H}]$ (as it belongs to any of the models $M_n[\mathcal{H}_n]$).  

\section{Subcompact cardinals}\label{section: subcompact cardinals}
In this section we will discuss the relationship between subcompactness of
cardinals and stationary reflection. Both of these concepts are related to Jensen's square principle.
\begin{definition}[Jensen]
Let $\kappa$ be a cardinal. A sequence $\mathcal{C} = \langle C_\alpha \mid \alpha < \kappa^{+}\rangle$ is a $\square_\kappa$-sequence if:
\begin{enumerate}
\item $C_\alpha$ is a closed unbounded subset of $\alpha$.
\item $\otp C_\alpha \leq \kappa$. 
\item For all $\beta \in \acc C_\alpha$, $C_\beta = C_\alpha \cap \beta$.
\end{enumerate}
\end{definition}
$\square_\kappa$ is a strong non-compactness principle, see \cite{Magidor2012Square} and \cite{CummingsForemanMagidor2001}. For example:
\begin{lemma}
Let $\kappa$ be a cardinal such that $\square_\kappa$ holds. For every stationary subset $S \subseteq \kappa^{+}$ there is a stationary subset $T \subseteq S$ that does not reflect.
\end{lemma}

Let $V$ be a model of $\ZFC$, such that $\kappa\in V$ is an infinite cardinal and $\square_\kappa$ holds. If $W$ is a larger model, $V \subseteq W$ and $(\kappa^{+})^V = (\kappa^{+})^W$ then $W \models \square_{|\kappa|^W}$. Thus, in order to obtain a model in which some type of stationary reflection holds at $\kappa^+$, without collapsing $\kappa^{+}$, we must start from a model in which $\square_\kappa$ fails.

The principle $\square_\kappa$ was originated from the study of the fine
structure of $L$. Jensen proved that $\square_\kappa$ holds in $L$ for all
infinite $\kappa$ and more sophisticated arguments provide square sequences in
larger core models. While studying the properties that imply the failure of
square, Jensen isolated the notion of subcompactness.

\begin{definition}[Jensen]
Let $\kappa$ be a cardinal. $\kappa$ is \emph{subcompact} if for every
$A\subseteq H(\kappa^{+})$ there is $\rho < \kappa$ and $B \subseteq
H(\rho^{+})$ such that there is an elementary embedding: \[j\colon \langle
H(\rho^{+}),\in,B\rangle \to \langle H(\kappa^{+}), \in ,A\rangle\] with $\crit
j = \rho$.
\end{definition}

Note that $j(\rho) = \kappa$, thus $\rho$ is $1$-extendible with target
$\kappa$.  Moreover, the set of ordinals $\delta < \kappa^{+}$ which are the
$\sup j\image \rho^{+}$ for some subcompact embedding with critical point $\rho$
is stationary.

The first subcompact cardinal is smaller than the first cardinal $\kappa$ which
is $\kappa^{+}$-supercompact, since the first subcompact is weakly compact but
not measurable.  Subcompact cardinals are weaker than supercompact cardinals,
but are still strong enough to imply the failure of $\square_\kappa$. 
Moreover, Zeman and Schimmerling proved, in \cite{SchimmerlingZeman2001}, that 
$\square_\kappa$ holds for every $\kappa$ which is not subcompact in models of the form $L[\vec{E}]$ which satisfy some modest iterability and solidity requirements. 

\begin{theorem}[Jensen]
If $\kappa$ is subcompact then $\square_\kappa$ fails.
\end{theorem}
\begin{proof}
Assume otherwise, and let $\mathcal{C}$ be $\square_\kappa$-sequence.  Let $\rho < \kappa$, and $\tilde{\mathcal{C}}$ be such that there is an elementary embedding:
\[j\colon \langle H(\rho^{+}), \in, \tilde{\mathcal{C}}\rangle \to \langle H(\kappa^{+}),\in, \mathcal{C}\rangle\]
Let $\mathcal{A}$ be $j\image \rho^{+}$. For every $\alpha \in \mathcal{A}$, if $\cf \alpha \neq \kappa$ then $\cf \alpha < \rho$ and therefore $\otp C_\alpha < \rho$. 

Let $\delta = \sup \mathcal{A}$, and let us look at $C_{\delta}$. Since $\cf \delta = \rho^{+}$, $\otp C_\delta \geq \rho^{+}$. On the other hand for every $\alpha \in \acc C_\delta$ with $\cf \alpha < \rho$, $\otp (C_\delta \cap \alpha) = \otp C_\alpha < \rho$, which is impossible.
\end{proof}


The same proof as above shows that the following stronger claim holds:
\begin{remark}[Zeman, \cite{Zeman2017}]
Let $\kappa$ be a subcompact cardinal. Then, there is no sequence $\langle \mathcal{C}_\alpha \mid \alpha \in S^{\kappa^{+}}_{<\kappa}\rangle$ such that:
\begin{enumerate}
\item For all $C \in \mathcal{C}_\alpha$, $C$ is a club at $\alpha$, $\otp C < \kappa$.
\item For all $\alpha \in S^{\kappa^{+}}_{<\kappa}$, $0 < |\mathcal{C}_\alpha| < \kappa$.
\item For all $C\in \mathcal{C}_\alpha$, $\beta\in \acc C$, $C \cap \beta \in \mathcal{C}_\beta$.
\end{enumerate}
\end{remark}

We note that $\square(\kappa^{+})$ can still hold where $\kappa$ is subcompact.
Indeed, subcompactness behaves much like Mahloness of $\kappa^{+}$, and cannot
be destroyed by a forcing which is $\kappa^{+}$-strategically closed.  The
forcing to add $\square(\kappa^+)$ is $\kappa^+$-strategically closed.  Yet, the
failure of $\square_\kappa$ for a subcompact $\kappa$ indicates that
subcompactness has a deep connection to stationary reflection. The following
argument (essentially due to Zeman) is similar to the Shelah-Harrington
\cite{HarringtonShelah1985} argument for obtaining
$\Refl(S^{\omega_2}_{\omega})$ from a Mahlo cardinal. In \cite{Zeman2017}, a
similar theorem is proven when the subcompact cardinal is collapsed to be
$\omega_n$ for some $n$. For completeness we include a proof here for the case
in which the subcompactness of $\kappa$ is preserved.

\begin{theorem}[Zeman]\label{bounded reflection at subcompact}
Let $\kappa$ be subcompact and assume that $2^{\kappa} = \kappa^{+}$ and let $\eta < \kappa$. Then, there is a forcing notion $\mathbb{P}$ that does not collapse cardinals and forces that every stationary subset of $S^{\kappa^{+}}_{<\eta}$ reflects at a point in $S^{\kappa^{+}}_{<\kappa}$ of arbitrary high cofinality. Moreover, $\kappa$ remains subcompact in the generic extension.
\end{theorem}
\begin{proof}
Let $\mathbb{Q}_0$ be an Easton support iteration of length $\kappa$. In the $\rho$ step, if $\rho$ is not inaccessible, force with the trivial forcing. Otherwise, force with $\Add(\rho^{+}, \rho^{++})$.

We define a forcing notion $\mathbb{P}$. $\mathbb{P}$ is essentially a $\kappa$-support iteration of length $\kappa^{++}$. Let us define, by induction on $\alpha < \kappa^{++}$, $\mathbb{P}_\alpha$ and $\mathbb{Q}_\alpha$. $\mathbb{Q}_0$ was already defined, and we let $\mathbb{P}_0$ be the trivial forcing and $\mathbb{P}_1 = \mathbb{Q}_0$.

Let $\alpha > 0$. Let us pick a name $\dot{S}_\alpha$ for a subset of $S^{\kappa^{+}}_{\theta}$ for some $\theta < \eta$. If there is $\bar{\kappa} < \kappa$ such that $\mathbb{P}_\alpha$ forces that $\dot{S}_\alpha$ does not reflect at any ordinal of cofinality between $\bar{\kappa}$ and $\kappa$, then we let $\mathbb{Q}_\alpha$ be the forcing that adds a club $C_\alpha$ disjoint from $S_\alpha$ by bounded conditions from $V^{\mathbb{Q}_0}$. 

For all $\gamma < \kappa^{++}$, let $\mathbb{P}_\gamma$ be the collection of all
sequences of length $\gamma$, $p$, such that $\supp p = \{\beta < \gamma \mid
p(\beta) \neq \emptyset\}$ has size at most $\kappa$ and for all $\beta <
\gamma$, $p\restriction \beta \Vdash_{\mathbb{P}_\beta} p(\beta)\in
\mathbb{Q}_\beta$.  We order $\mathbb{P}_\gamma$ in the natural way.  Let
$\mathbb{P} = \mathbb{P}_{\kappa^{++}}$.

Since our forcing notions are going to be only distributive and not closed (or strategically closed), we wish to avoid the delicate point of whether the conditions from $V^{\mathbb{Q}_0}$ are a dense subset of the iteration, and thus $\mathbb{P}$ is not defined as the standard iteration of the $\mathbb{Q}_\alpha$.

\begin{lemma}\label{lemma: distributiveness}
Every $\mathbb{P}_\alpha$-name for a $\kappa$-sequence of ordinals is forced to be a $\mathbb{Q}_0$-name.
\end{lemma}
\begin{proof}
Let us prove the lemma by induction. For $\alpha = 1$, $\mathbb{P}_\alpha = \mathbb{Q}_0$ and the statement is trivial.

Let us assume now that the claim is true for all $\beta < \alpha$. Since
$|\mathbb{P}_\alpha| \leq \kappa^{+}$, we can code $\mathbb{P}_\alpha$ as a
subset of $H(\kappa^{+})$.  Let $\rho < \kappa$ such that there is an elementary
embedding:

\[j\colon \langle H(\rho^{+}), \in, \tilde{\mathbb{P}}_{\tilde{\alpha}}\rangle
\to \langle H(\kappa^{+}), \in, \mathbb{P}_{\alpha}\rangle.\]

By elementarity, $\tilde{\mathbb{P}}_{\tilde{\alpha}}$ codes an iteration for
killing non reflecting subsets of $\rho^{+}$, of length $\tilde{\alpha}$ in the
same way as $\mathbb{P}_\alpha$.  Let us denote the components of the iteration
by $\tilde{\mathbb{Q}}_{\beta}$.  We note that if $H_0$ is
$\mathbb{Q}_0$-generic, then $j$ lifts to the extension of $H(\rho^+)$ by $H_0
\upharpoonright \rho$.  In particular we can apply the elementarity of $j$ to
the coordinates $\tilde{\mathbb{Q}}_\beta$.

We build a generic filter $\tilde{G}$ for $\tilde{\mathbb{P}}_{\tilde{\alpha}}$
using the Cohen generic subsets of $\rho^+$ added by the iteration
$\mathbb{Q}_0$.  In fact we will show that $\tilde{\mathbb{P}}_\alpha$ is
equivalent to the Cohen forcing of subsets of $\rho^+$ over
$H(\rho^+)[H_0 \upharpoonright \rho]$.

To do so we define clubs $E_\gamma$ and $C_\gamma$ for $\gamma< \tilde{\alpha}$
with the following properties. Using the clubs $E_\gamma$ for all $\gamma < \tilde{\alpha}$, we can show that for all $\beta \leq \tilde{\alpha}$, $\tilde{\mathbb{P}}_\beta$ is equivalent (externally to $H(\rho^+)$) to the forcing that adds
$\beta$ Cohen subsets of $\rho^+$. We do this by taking the dense set of conditions
$p$ such that for all $\gamma<\beta$, either $p(\gamma) = \emptyset$ or
$\max(p(\gamma)) \in E_\gamma$.

We will also construct clubs $C_\gamma$ for all $\gamma < \tilde{\alpha}$. Those clubs are going to be $H(\rho^+)$-generic in the following sense. The natural filter $\tilde{G}_\beta\subseteq
\tilde{\mathbb{P}}_\beta$ given by the set of all $p$ in
$\tilde{\mathbb{P}}_\beta$ such that $p(0) \in H_0 \upharpoonright  \rho$ and

\[\forall 0 < \gamma < \beta,\, q(\gamma) = \emptyset \vee q(\gamma) = C_\gamma
\cap \max (q(\gamma) + 1)\]

is $H(\rho^+)$-generic.

So it remains to construct the clubs $E_\gamma$ and $C_\gamma$ and prove that
$\tilde{G}_\beta$ is generic.  We go by induction on $\beta$. Suppose that we
have constructed $E_\gamma$ and $C_\gamma$ for all $\gamma<\beta$.

If $\beta = \zeta +1$, then by elementarity and the definition of the iteration
we have that $j(\dot{\tilde{S}}_\zeta)= \dot{S}_{j(\zeta)}$ is forced by
$\mathbb{P}_\zeta$ to be a set consisting of ordinals of cofinality less than
$\rho$ which does not reflect at $\delta = \sup j \image \rho^+$.  Again by
elementarity, we can interpret $\dot{S}_{j(\zeta)} \cap j``\rho^+$ using only
$\tilde{G}_\zeta$.

By induction $\tilde{\mathbb{P}}_\gamma$ is equivalent to adding Cohen subsets
of $\rho^+$ and using a straightforward density argument it follows that there
is a dense subset of $p$ in $\tilde{\mathbb{P}}_\gamma$ such that each
nontrivial coordinate of $p$ has the same maximum element.  It follows that the
condition $m$ in $\mathbb{P}_{j(\zeta)}$ given by $\supp(m) = \bigcup_{q\in
\tilde{G}_\zeta} \supp j(q)$ and for all $\gamma \in \supp(m)$, $m(\gamma) =
\bigcup_{q\in\tilde{G}_\zeta} j(q)(\gamma) \cup \{\delta\}$ is a master
condition for $j``\tilde{G}_\zeta$.

It follows that $m$ decides $\dot{S}_{j(\zeta)} \cap j \image \rho^+$.  This set
is nonstationary in $V[H_0 \restriction \rho][\tilde{G}_\zeta]$, since otherwise it would
remain stationary in the full generic extension.  It follows that we can find a
club $D_\zeta$ in $\delta$ which is disjoint from it.  Let $E_\zeta = \{ \xi <
\rho^+ \mid j(\xi) \in D_\zeta \}$.  If we consider the dense subset of
$\tilde{\mathbb{Q}}_\zeta$ whose maximum element is in $E_\gamma$, then this forcing is
isomorphic to adding a Cohen subset of $\rho^+$ over $V[H_0\upharpoonright
\rho][\tilde{G}_\zeta]$. We stress that this isomorphism can be computed 
in the model $V[H_0\upharpoonright \rho][\tilde{G}_\zeta]$. 
Let $X_\zeta$ be a Cohen subset of $\rho^+$ which is generic over $V[H_0\upharpoonright \rho][\tilde{G}_\zeta]$ (there are such sets since $\tilde{G}_\zeta$ is equivalent to
a subset of $\zeta^+$ and $\Add(\zeta^+,\zeta^{++})$ is $\zeta^{++}$-cc). Applying the isomorphism between $\Add^{V[H_0\restriction \rho]}(\rho^+,1)$ and $\tilde{\mathbb{Q}}_\zeta$ on $X_\zeta$ 
we obtain a club $C_\zeta$
which is $\tilde{\mathbb{Q}}_\zeta$-generic over $V[H_0 \upharpoonright
\rho][\tilde{G}_\zeta]$. In fact $V[H_0 \upharpoonright
\rho][\tilde{G}_\zeta][C_\zeta] = V[H_0 \upharpoonright
\rho][\tilde{G}_\zeta][X_\zeta]$.  This completes the successor step.

If $\beta$ is limit, then using $E_\gamma$ for $\gamma<\beta$ and induction
there is an isomorphism between $\tilde{\mathbb{P}}_\beta$ and
$\Add(\rho^+,\beta)$ as computed in $V[H_0\upharpoonright \rho]$.  The fact that
 the sequence $\langle X_\gamma \mid \gamma < \beta \rangle$ is generic for
$\Add(\rho^+,\beta)$ implies that $\tilde{G}_\beta$ is generic for
$\tilde{P}_\beta$.

We conclude that there is a generic filter $\tilde{G}$ for
$\tilde{\mathbb{P}}_{\tilde{\alpha}}$. This generic filter is obtained in a
$\rho^{+}$-distributive extension of $V^{\mathbb{Q}_0 \upharpoonright \rho}$.
Thus, it does not introduce any new $\rho$-sequences of ordinals (recall that
$H(\rho^{+})$ is closed under $\rho$-sequences and thus computes correctly
$\rho$-distributivity). The lemma follows by elementarity.  \end{proof}

By the chain condition of $\mathbb{P}$, we can make sure that in the generic
extension, if $S$ is a subset of $S^{\kappa^{+}}_{<\kappa}$ which does not
reflect then $\Vdash S = \dot{S}_\alpha$ for some $\alpha < \kappa^{++}$ and therefore it is nonstationary.

\begin{lemma}
$\kappa$ is subcompact in the generic extension. 
\end{lemma}

\begin{proof}
Let $p$ be a condition and let $A$ be a name for a subset of $H(\kappa^{+})$.
By the chain condition of the iteration, there is an $\alpha < \kappa^{++}$ such
that $\dot{A}$ is a $\mathbb{P}_\alpha$-name.

By the subcompactness of $\kappa$ there is a cardinal $\rho < \kappa$ and an elementary embedding 
\[j\colon \langle H(\rho^{+}), \in, \bar{p}, \bar{\mathbb{P}}_{\bar{\alpha}},
\bar{A}\rangle \to \langle H(\kappa^{+}), \in, p, \mathbb{P}_\alpha, A\rangle.\]

By the arguments of Lemma \ref{lemma: distributiveness}, we can find a master condition $m$, namely a condition $m \leq p$ such that for any dense open set $D \subseteq \bar{\mathbb{P}}$ which is definable from parameters in $H(\rho^{+})$ and $A$, there is $q \in D$ such that $m \leq j(q)$. It is clear that in this case, if $m$ belongs to the generic filter than $j$ lifts to the generic extension.

Our argument shows that the set of such master conditions is dense in
$\mathbb{P}_\alpha$, so the lemma follows.\end{proof}

This finishes the proof of Theorem \ref{bounded reflection at subcompact}.\end{proof}

The next theorem shows that it is consistent that $\kappa$ is subcompact yet there is no forcing extension that preserves $\kappa$ and $\kappa^{+}$ and forces full stationary reflection at $S^{\kappa^{+}}_{<\kappa}$.
\begin{theorem}\label{theorem: subcompact is not enough}
Let $\kappa$ be subcompact. There is a generic extension in which $\kappa$ is subcompact and there is a non reflecting stationary set $S \subseteq S^{\kappa^{+}}_{<\kappa}$ and a partial square $\langle C_\alpha \mid \alpha \notin S\rangle$.
\end{theorem}
\begin{proof}
By preparing the ground model, if necessary, we may assume that for every $\delta < \kappa$, which is not subcompact $\square_{\delta}$ holds.
 
Let $\mathbb{P}$ be a forcing notion which consists of pairs $\langle s, c\rangle$ where:
\begin{enumerate}
\item $s$ is a bounded subset of $S^{\kappa^{+}}_{<\kappa}$ and for all limit
$\alpha \leq \sup s$, there is a club $d_\alpha$ in $\alpha$, which is disjoint from $s$. 
\item If $\beta \in s$ then $\cf \beta$ is non-measurable.
\item $c$ is a function and $\dom c$ is a successor ordinal between $\sup s$ and $\kappa^{+}$.
\item For every $\alpha \in \dom c$, $c(\alpha)$ is a closed subset of $\alpha$ (possibly the empty set).
\item If $\alpha\in \dom c \setminus s$ then $\sup c(\alpha) = \alpha$.
\item If $\beta \in \acc c(\alpha)$ then $c(\alpha)\cap \beta = c(\beta)$. 
\end{enumerate}

We order $\mathbb{P}$ by $\langle s', c'\rangle \leq \langle s, c\rangle$ if and
only if $s \subseteq s'$, $c = c' \restriction \dom c$ and $(s' \setminus s) \cap (\dom c) = \emptyset$ (note that $s'$ is an end extension of $s$ above the maximum of the domain of $c$, which is at least $\sup s$).

\begin{claim}
The forcing $\mathbb{P}$ is $\kappa + 1$-strategically closed.
\end{claim}
\begin{proof}  We define a winning strategy for the good player.  At successor
stages, the good player does nothing.  At limit stages, if the current stage of
the game is $\langle \langle s_\alpha, c_\alpha\rangle \mid \alpha <
\beta\rangle$, then setting $\rho_\alpha = \max \dom c_\alpha$ the good player
plays

\[s_{\beta} = \bigcup_{\alpha < \beta} s_\alpha \text{ and }c_\beta = \bigcup_{\alpha < \beta} c_\alpha \cup \{\langle \rho_{\beta}, \{\rho_\alpha \mid \alpha < \beta\}\rangle\}.\]

It is clear that this choice is a condition in $\mathbb{P}$ which is stronger
than all previous conditions in the play provided that $\beta \leq \kappa$.
\end{proof}

By the proof of the claim, it is clear that the strategy is definable in $H(\kappa^{+})$. Moreover, throughout the game the ordinals $\rho_\alpha$ will be a club which witness the non reflection of $s$ at each limit point.

Let us show that $\kappa$ is subcompact in the generic extension. Let $\dot{A}$
be a name for a subset of $H(\kappa^{+})$ in the generic extension. Since
$\mathbb{P} \subseteq H(\kappa^{+})$, we have $\dot{A} \subseteq H(\kappa^{+})$.

Let $\rho < \kappa$ and $\dot{B}, \tilde{\mathbb{P}} \subseteq H(\rho^{+})$ be such that there is an elementary embedding:
\[j\colon \langle H(\rho^{+}), \in, \tilde{\mathbb{P}}, \dot{B}\rangle \to \langle H(\kappa^{+}), \in, \mathbb{P}, \dot{A}\rangle.\]

Moreover, let us assume that $\rho$ is the minimal cardinal for which such $\dot{B}$ and $j$ exist. 
\begin{claim}
$\rho$ is not subcompact.
\end{claim}
\begin{proof}
Assume that $\rho$ is subcompact. Then there is some $\eta < \rho$ and $\dot{C}, \tilde{\tilde{\mathbb{P}}}$ and an elementary embedding $k$ such that:
\[k \colon \langle H(\eta^{+}), \in, \tilde{\tilde{\mathbb{P}}}, \dot{C}\rangle \to \langle H(\rho^{+}), \in, \tilde{\mathbb{P}}, \dot{B}\rangle.\]

Then $j\circ k$ is elementary, which contradicts the assumption that $\rho$ is minimal.
\end{proof}
The forcing $\tilde{\mathbb{P}}$ is $\rho + 1$-strategically closed by
elementarity. By a theorem of Ishiu and Yoshinobu \cite{Ishiu2002}, since $\square_\rho$ holds, $\tilde{\mathbb{P}}$ is $\rho^{+}$-strategically closed. This strategy is combined from the strategies for the shorter games and thus we can verify that the sequence of $\rho_\alpha$ which is constructed in the game is closed and disjoint from the constructed non-reflecting set.
Let \[\mathcal{D} = \{ D_{\varphi, a} \mid \varphi(x,y) \text{ is a first order formula}, a \in H(\rho^{+})\}\] 
be the set of all dense open subsets of $\tilde{\mathbb{P}}$ of the form 
\[D_{\varphi, a} = \{p\in\tilde{\mathbb{P}} \mid p\Vdash \neg\varphi(a, \dot{B}) \text{ or } p \Vdash \varphi(a, \dot{B})\}.\]

Let $\langle D_\alpha \mid \alpha < \rho^{+}\rangle$ be an enumeration of $\mathcal{D}$ with length $\rho^{+}$. Using the strategic closure of $\tilde{\mathbb{P}}$ we can generate a decreasing sequence of conditions $\langle p_\alpha \mid \alpha < \rho^{+}\rangle$ such that $p_\alpha \in D_\alpha$. 

Let $\tilde{G}\subseteq \tilde{\mathbb{P}}$ be the filter generated from the sequence $\langle p_\alpha \mid \alpha < \rho^{+}\rangle$. Let us show that there is a condition $m\in\mathbb{P}$ such that $\forall q\in\tilde{G}, m \leq j(q)$. This implies that $m$ forces that the embedding $j$ lifts to the generic extension. 

Indeed, let $p_\alpha = \langle s_\alpha, c_\alpha\rangle$. Then clearly, for $\alpha < \beta$, $s_\alpha$ is an initial segment of $s_\beta$ and $c_\alpha$ is an initial segment of $c_\beta$. Let $\delta = \sup j\image \rho^{+}$ and let us consider 
\[s = \{\delta\} \cup \bigcup_\alpha j(s_\alpha)\]
\[c = \bigcup_\alpha j(c_\alpha) \cup \{\langle \delta, \emptyset\rangle\}\]
The strategy enables us to obtain a club $E \subseteq \rho^{+}$ which is disjoint from $\bigcup_{\alpha} s_\alpha$. $j\image E$ is disjoint from $s$. Moreover, the closure of $j\image E$ differs from $j\image E$ only by points of cofinality $\rho$. Since $\rho$ is measurable those points cannot appear at $s$ and therefore also $\acc j\image E$ is disjoint from $s$.
\end{proof}

The theorem suggests that the consistency of full stationary reflection at a subcompact cardinal might exceed the consistency of subcompact cardinal. Moreover, since the forcing is $\kappa^{+}$-distributive, we can conclude that if $\kappa$ is measurable subcompact or even more it will remain measurable subcompact after the forcing and there is no generic extension in which stationary reflection holds at $S^{\kappa^{+}}_{<\kappa}$ and $\kappa, \kappa^{+}$ are preserved.

The exact large cardinal assumption which is required in order to get stationary reflection at the set $S^{\kappa^{+}}_{<\kappa}$ where $\kappa$ is subcompact is unclear. In the previous theorem, the set of all $\beta < \kappa^{+}$ such that there is an elementary embedding $j\colon H(\rho^{+}) \to H(\kappa^{+})$ with $\sup j\image \rho^{+} = \beta$ is stationary and non-reflecting. This is analogous to the case of Mahlo cardinal in a generic extension of $L$ in which stationary sets of bounded cofinality might reflect at inaccessible cardinals but the set of inaccessible cardinals does not reflect. 

A similar forcing argument as in Theorem \ref{theorem: subcompact is not enough}
shows that if $V$ is any model of $\ZFC + \GCH$ then there is a generic
extension in which for every cardinal $\kappa$, there is a coherent sequence
$\langle C_\alpha \mid \alpha < \kappa^{+}\rangle$, $\otp C_\alpha \leq \kappa$,
such that $C_\alpha$ is a club at $\alpha$ if and only if there is no elementary embedding $j\colon H(\rho^{+}) \to H(\kappa^{+})$ with $\alpha = \sup j\image \rho^{+}$. Let us assume that such a partial square exists, and that stationary reflection holds at $S^{\kappa^{+}}_{<\kappa}$. Then every stationary set reflects at ordinals $\delta$ which are $\sup j\image \rho^{+}$ for some elementary $j\colon H(\rho^{+})\to H(\kappa^{+})$. 

The following definition, due to Neeman and Steel, will play a major role in our investigation of unbounded stationary reflection. In their paper,
\cite{NeemanSteelSubcompact}, this large cardinal notion is denoted by
$\Pi^2_1$-subcompact. We feel that the notion $\kappa^+$-$\Pi^1_1$-subcompact is more appropriate as
it emphasizes the resemblance between $\kappa^+$ and a weakly compact cardinal.

\begin{definition}\label{definition: pi11-subcompact}
A cardinal $\kappa$ is $\kappa^+$-$\Pi^1_1$-subcompact if for every set $A\subseteq H(\kappa^{+})$, and every $\Pi^1_1$-statement $\Phi$ such that $\langle H(\kappa^{+}), \in, A\rangle \models \Phi$, there is $\rho < \kappa$, $B\subseteq H(\rho^{+})$ and an elementary embedding:
 \[j\colon \langle H(\rho^{+}), \in, B\rangle \to \langle H(\kappa^{+}), \in, A\rangle\]
such that $\langle H(\rho^{+}), \in, B\rangle\models \Phi$.
\end{definition}
\begin{lemma}\label{lemma: pi11-subcompact is measurable}
Let $\kappa$ be $\kappa^+$-$\Pi^1_1$-subcompact. Then $\kappa$ is measurable.
\end{lemma}
\begin{proof}
Let $\Phi$ be the $\Pi^1_1$-statement ``for every $\mathcal{U} \subseteq \power(\kappa)$ which is an ultrafilter, $\mathcal{U}$ is not $\kappa$-complete''. If $\kappa$ is not measurable, $\Phi$ holds. But for every $\rho < \kappa$ such that there is an elementary embedding $j\colon H(\rho^{+}) \to H(\kappa^{+})$ with critical point $\rho$, one can obtain a measure of $\rho$ by $\mathcal{U}_{\rho} = \{A \subseteq \rho \mid \rho\in j(A)\}$. So $\Phi$ fails at $H(\rho^{+})$. 
\end{proof}
\begin{lemma}\label{lemma: reflection at pi11 subcompact}
Let $\kappa$ be $\kappa^+$-$\Pi^1_1$-subcompact. Then every sequence of $<\kappa$ many stationary subsets of $S^{\kappa^{+}}_{<\kappa}$ has a common reflection point. 
\end{lemma}
\begin{proof}
Let $\mathcal{S}$ be a collection of stationary sets, $|\mathcal{S}| < \kappa$.  Let us reflect the $\Pi^1_1$-statement: ``$\forall C\subseteq \kappa^{+}$, which is closed and unbounded, for all $S\in \mathcal{S}$, $C \cap \mathcal{S} \neq \emptyset$''. 

Fix $\rho > |\mathcal{S}|$ such that there is an elementary embedding:
\[j\colon \langle H(\rho^{+}), \in, \tilde{\mathcal{S}}\rangle \to \langle
H(\kappa^{+}), \in, \mathcal{S} \rangle\]

Note that if $S\in\mathcal{S}$, and $\alpha \in S$ with $\cf(\alpha) = \eta$,
then $\eta < \rho$. The ordinal $\delta = \sup j\image \rho^{+}$ will be a reflection point of every member of $\mathcal{S}$. Indeed, for every $S\in\mathcal{S}$ there is a unique $\tilde{S}\in\tilde{\mathcal{S}}$ such that $j(\tilde{S}) = S$. Every $\tilde{S}\in\tilde{\mathcal{S}}$ is stationary at $\rho^{+}$ of cofinality $<\rho$. Therefore, $j\image\tilde{S} = S \cap j\image \rho^{+}$ is stationary at $\delta$.  
\end{proof}
Neeman and Steel showed that the consistency strength of simultaneous stationary
reflection at the successor of a Woodin cardinal (indeed, threadable successor of
a Woodin cardinal) is $\kappa^+$-$\Pi^1_1$-subcompact under some iterability assumptions.

By analogy with the case of consistency strength of various assertions about stationary reflection which are around the existence of a Mahlo cardinal, we expect the consistency strength of full stationary reflection at a subcompact cardinal to be strictly between a subcompact cardinal and a $\kappa^+$-$\Pi^1_1$-subcompact cardinal.

\section{Stationary Reflection at $\aleph_{\omega+1}$}\label{section: stationary reflection}

In this section, we will prove the main theorem of the paper which improves the
upper bound of the consistency strength of stationary reflection at the successor of a singular cardinal. The
proof splits into two components: the first component is a general statement
about preservation of some mildly indestructible reflection principles at
successor of a measurable cardinal $\kappa$ under a forcing that changes the
cofinality of $\kappa$ to $\omega$ and shoots a club through the stationary set
$(S^{\kappa^{+}}_{\kappa})^V$.  The second is to show how to obtain the
hypothesis of the previous result from $\Pi_1^1$-subcompactness.  We formulate
the result this way, since it may be possible to use a weaker large cardinal
notion to obtain the hypothesis of the first step.  Further, we are able to use
the first step to give an application to stationary reflection for subsets of
some bounded cofinality.

The idea to use Prikry forcing in order to force a measurable to be
$\aleph_\omega$ while preserving its successor and maintaining stationary
reflection at almost all stationary subsets of $\aleph_{\omega + 1}$ appears in
several places. In the paper of Cummings, Foreman and Magidor \cite[Section
11]{CummingsForemanMagidor2001}, they show that if $\kappa$ is
$\kappa^{+}$-supercompact then after forcing with Prikry forcing stationary
reflection holds outside the set of ordinals of ground model cofinality
$\kappa$. This was later extended by Faubion in \cite{Faubion2012} to force from
$\kappa$ which is $\kappa^{+}$-supercompact $\Refl(\aleph_{\omega+1}\setminus
S_0)$ where $S_0$ is the collection of ordinals with ground model cofinality
$\kappa$.  Clearly to obtain full stationary reflection in such a model, we must
destroy the stationarity of $S_0$ and any other nonreflecting stationary sets
which may appear in the extension.

We will state and prove the main theorem for simultaneous reflection of finitely many 
stationary sets. The proof adjusts easily to the case of stationary reflection of single sets.

\begin{theorem}\label{theorem: preserving stationary reflection}
Assume $\GCH$. Let $\kappa$ be a measurable cardinal and let $S \subseteq S^{\kappa^{+}}_{<\kappa}$ be stationary. Let us assume that for every $\omega \leq \theta < \mu < \kappa$ regular cardinals, $S^\mu_\theta \in I[\mu]$.

Let us assume further that $\Add(\kappa^{+},1)$ forces that every finite
sequence of stationary subsets of $S$ reflects simultaneously at ordinals of
unbounded cofinalities below $\kappa$. Then, there is a generic extension in
which $\kappa = \aleph_{\omega}$, $\kappa^{+} = \aleph_{\omega + 1}$, the ground
model $S^{\kappa^{+}}_{\kappa}$ is nonstationary and simultaneous reflection
holds for finite sequences of stationary subsets of $S$.
\end{theorem}

By the assumption of $\GCH$, the approachability requirement is not satisfied trivially only at successors of singular cardinals.
\begin{proof}
In $V$, let us fix a normal ultrafilter on $\kappa$, $\mathcal{U}$. Let $j:V \to
M$ be the ultrapower embedding given by $\mathcal{U}$.  Let $\mathbb{P}$ be the
Prikry forcing with interleaved collapses using a guiding generic $K$ as defined
in Section \ref{section: Prikry}.

Let $\dot{P}$ be the canonical name of the generic Prikry sequence added by $\mathbb{P}$ (so $\dot{P}$ does not include the generic filters for the collapses). Let $\mathbb{Q}$ be the forcing notion for adding a club to $(S^{\kappa^{+}}_{<\kappa})^V$ in  $V[\dot{P}]$. Note that $\mathbb{Q}$ is defined in a submodel of the generic extension of $V$ by $\mathbb{P}$.

Let us start by analyzing $j_{\omega}(\mathbb{Q})$.
\begin{claim}
$\cf^V j_{\omega}(\kappa^{+}) = \kappa^{+}$. Moreover, in $V$ there is a closed unbounded set $D \subseteq j_\omega(\kappa^{+})$ such that $D \subseteq (S^{j_\omega(\kappa^{+})}_{\leq \kappa})^{M_{\omega}}$. 
\end{claim}
\begin{proof}
The sequence $j_\omega\image \kappa^{+}$ is cofinal at $j_{\omega}(\kappa^{+})$. Indeed, let $\beta < j_\omega(\kappa^{+})$. Then there is a function $f\colon \kappa^n \to \kappa^{+}$ such that $\beta = j_{\omega}(f)(\kappa, j_1(\kappa), \dots, j_{n-1}(\kappa))$. Let $\gamma = \sup_{a\in \kappa^n} f(a) < \kappa^{+}$. Then $j_{\omega}(\gamma) > \beta$.

Let $D = \acc j_\omega \image \kappa^{+}$. For every $\delta \in D$, let us show
that $\cf^V\delta = \cf^{M_\omega} \delta$. Clearly, $\cf^V\delta \leq
\cf^{M_{\omega}}\delta$. Let us assume that $\cf^V\delta = \eta \leq \kappa$ and
let $\langle \gamma_i \mid i < \eta\rangle$ be a sequence of ordinals such that
$\sup_{i < \eta} j_\omega(\gamma_i) = \delta$. The sequence $\langle
j_{\omega}(\gamma_i) \mid i < \eta\rangle$ belongs to $M_{\omega}$ since it is
$j_\omega(\langle \gamma_i\mid i < \eta\rangle) \restriction \eta$. Therefore, $M_{\omega}$ computes the cofinality of $\delta$ correctly.
\end{proof} 

Let $p_\star \in \mathbb{P}$ and let $s_\star = \langle \rho_0, \dots, \rho_{n_\star - 1}\rangle$ be the Prikry part of the stem of $p_\star$.

Let $D$ be as in the conclusion of the claim. Let $P = s_\star ^\smallfrown \langle j_n(\kappa) \mid n < \omega\rangle$. By a theorem of Mathias, $P$ is an $M_\omega$-generic Prikry sequence. So one can think of $P$ as the realization of $j_\omega(\dot{P})$ using the generic filter over $M_{\omega}$ which is obtained from Lemma \ref{lemma: prikry generic over Momega}.

Since $|j_{\omega}(\mathbb{Q})^P|^V = \kappa^{+}$, in $V$ one can construct a tree of conditions in $j_{\omega}(\mathbb{Q})^P$ which is isomorphic to $(\kappa^{+})^{{<}\kappa^{+}}$. This is done by induction. Assume that for $\eta\in (\kappa^{+})^{{<}\kappa^{+}}$, $q_{\eta}$ is defined. Let $\langle r^\eta_\alpha \mid \alpha < \kappa^{+}\rangle$ enumerate all conditions in $j_{\omega}(\mathbb{Q})^P$ which are stronger than $q_{\eta}$. For each $\alpha < \kappa^{+}$, let $q_{\eta^\smallfrown \langle \alpha\rangle}$ be an extension of $r^\eta_\alpha$, such that $\max q_{\eta^\smallfrown \langle \alpha\rangle} \in D$.  If $\eta \in (\kappa^{+})^{{<}\kappa^{+}}$, and $\len \eta$ is a limit ordinal, we let $q_{\eta} = \{\delta_\eta\} \cup \bigcup_{\gamma < \len \eta} q_{\eta \restriction \gamma}$ where $\delta_{\eta} = \sup \{ \max q_{\eta \restriction \gamma} \mid \gamma < \len \eta\}$. Note that $\delta_{\eta}\in D$ since $D$ is club. Moreover, $q_{\eta} \in M_{\omega}[P]$, since this model is closed under $\kappa$-sequences from $V$.

Therefore, in $V$, there is a tree which is dense in $j_{\omega}(\mathbb{Q})^P$ and isomorphic to the forcing $\Add(\kappa^{+}, 1)^{V}$.  


Let us remark that if $P'$ is any other Prikry sequence such that $P'$ differs from $P$ by only finitely many ordinals, then the interpretation of $\mathbb{Q}$ is the same. In particular, for every $n$, $j_n(j_{\omega}(\mathbb{Q})^P) = j_{\omega}(\mathbb{Q})^P$.

Fix a $V$-generic for $\mathbb{P}_n \restriction p_\star$, $G'$, and let
$\tilde{G}$ be the $M_{\omega}$-generic filter for $j_{\omega}(\mathbb{P})$,
which is derived from it.  Let $H$ be a $V$-generic filter for
$j_{\omega}(\mathbb{Q})^P$ where $P$ derived from $\tilde{G}$.

We will apply the machinery of Subsection \ref{subsection: curly H} for
$\mathbb{A} = j_{\omega}(\mathbb{Q})^P$.  Let $\mathcal{H}$ be $\langle
{<}j_{n,\omega}\image H {>} \mid n < \omega\rangle$ and let $\mathcal{H}_n =
\langle {<} j_{m, n} \image H {>} \mid m \leq n\rangle$.  Recall that $P \ast
\mathcal{H}$ is generic for the forcing $j_{\omega}(\mathbb{H})$ over
$M_{\omega}$, and that $H \in M_\omega[P][\mathcal{H}]$ by Lemma \ref{lemma:
Momega H is intersection}.

\begin{claim}\label{claim: stationary reflection in Momega}
In $M_{\omega}[\tilde{G}][\mathcal{H}]$ every finite collection of stationary
subsets of $j_{\omega}(\kappa^{+})$ reflects at a common point.
\end{claim}
\begin{proof}
For $i \leq k$, let $S_i \subseteq j_{\omega}(\kappa^{+})$ be stationary. Since $H \in M_{\omega}[P][\mathcal{H}]$, we may assume that
each $S_i$ is disjoint from the set of ordinals that has cofinality
$j_{\omega}(\kappa)$ in $M_{\omega}$. Without loss of generality, the cofinality
of the members of each $S_i$ is fixed to be some $\theta_i < \kappa_m$.

Work in $M_n[\tilde{G} \restriction (n_\star + n + 1)][\mathcal{H}_k]$, $n \geq
m$. In this model, one can construct $\tilde{G}$ and $\mathcal{H}$ as well as $M_{\omega}$. 

Let \[T_n^i = \{\alpha < j_n(\kappa^{+}) \mid j_{n,\omega}(\alpha) \in \dot{S_i}^{\tilde{G} \ast \mathcal{H}}\}.\]

If the sequence $T_n^i$ for $i \leq k$ are all stationary in $j_n(\kappa^{+})$
then, since the forcing that introduces $\tilde{G}\restriction (n_\star + n + 1)$
has cardinality $j_n(\kappa)$ in $M_n[\mathcal{H}_n]$, one can find stationary
subsets of $T_n^i$, $T^i$, in $M_n[\mathcal{H}_n]$. Since $\mathcal{H}_k$ is
equivalent to a generic filter for $\Add(j_n(\kappa^+), n + 1)$ over $M_n$,
simultaneous stationary reflection holds in this model. In particular, the
sequence $T^i$ for $i \leq k$ reflects at common ordinals of arbitrary large
cofinalities. Let $\delta$ be a reflection point of $T$ such that $\kappa_n >
\cf \delta \geq \kappa_{n-1}$.  Recall that for all $i \leq k$, $T^i$ consists
of ordinals of cofinality $\theta_i$ which is less than $\kappa_m$ (in
particular less than $\kappa_{n-1}$).  Let $\{\beta_i \mid i < \cf \delta\}$ be
a continuous and increasing sequence of ordinals, cofinal at $\delta$. Let $A^i
= \{\gamma < \cf \delta \mid \beta_\gamma \in T^i\}$. By the assumption, $A^i$
is stationary in $S^{\cf \delta}_{\theta}$.

The forcing that introduces $\tilde{G} \restriction (n_\star + n + 1)$, splits
into a product of $\cf \delta$-cc forcing and $\theta^{+}$-closed forcing.
Using the approachability assumption $\cf \delta \in I[\cf \delta]$, we conclude
that each $A^i$ is stationary in $M_n[\tilde{G} \restriction (n_\star + n + 1)]$
and in particular in $M_n[\tilde{G} \restriction (n_\star + n +
1)][\mathcal{H}_n]$. The set $j_{k,\omega}(A^i) = j_{k,\omega}\image A^i$ belongs to $M_{\omega}$ and by downwards absoluteness from $M_n[\tilde{G} \restriction (n_\star + n +
1)][\mathcal{H}_n]$, it is stationary in $M_{\omega}[\tilde{G}][\mathcal{H}]$. 

Thus, if each $T_n^i$ is stationary then there is a condition that forces that
the sequence of $S_i$ reflect at a common point. Therefore, we conclude that at
least one of the  $T_n^i$ is non-stationary. Let $C_n$ be a club in
$M_n[\tilde{G}\restriction (n_{\star} + n + 1)][\mathcal{H}_n]$ disjoint from
$T_n^{i_n}$ for the relevant $i_n \leq k$.  By the chain condition of the
forcing that introduces $\tilde{G}\restriction (n_{\star} + n + 1)$, we may
assume that $C_n\in M_n[\mathcal{H}_n]$. Let $\dot{C}_n$ be a name for the club
$C_n$.

Let us consider $C = \bigcap_{n > m}
j_{n,\omega}(\dot{C}_n)^{{<}j_{n,\omega}\image \mathcal{H}_n{>}}$.  We claim
that $C \in M_{\omega}[P][\mathcal{H}]$. Indeed, for each $n$ the filter
${<}j_{n,\omega}\image \mathcal{H}_k{>}$ is simply an initial segment of
$\mathcal{H}$, $j_{n,\omega}(\dot{C}_n)$ is a member of $M_{\omega}$ and
$M_{\omega}[P][\mathcal{H}]$ is closed under $\omega$-sequences.

Let $i \leq k$ be such that $i_n = i$ for infinitely many $n \geq m$.  Let us
show that $C$ is disjoint from $S_i$. Indeed, if $\alpha \in C \cap S_i$ then
$\alpha = j_{n,\omega}(\alpha')$ for some $n < \omega$. Without loss of
generality, we can take $n$ to be such that $i_n=i$. Then $\alpha'\in C_n$ and
in $T_n^i$, a contradiction to the choice of $C_n$.
\end{proof}

By elementarity, we conclude that when forcing over $V$ with $\mathbb{P} \ast
\mathbb{H}/\dot{P}$ simultaneous stationary reflection holds at $\kappa^{+}$ for
finite collections.  \end{proof}

The proof shows that the forcing $\mathbb{P} \ast \mathbb{H}/\dot{P}$ preserves the stationarity of subsets of $\kappa^+$. In particular, the conclusion is never vacuous, as the set $S$ for which $\Refl(S)$ holds is stationary in the generic extension.

In order to use Theorem \ref{theorem: preserving stationary reflection}, we need to show that some indestructibility can be achieved at the level of subcompact cardinals.

\begin{lemma}\label{lemma: indestructible pi11 subcompact}
Let $\kappa$ be $\kappa^+$-$\Pi^1_1$-subcompact. There is a generic extension in which $\GCH$ holds, $\kappa$ is $\kappa^+$-$\Pi^1_1$-subcompact and this property is indestructible under the forcing $\Add(\kappa^{+},1)$.
\end{lemma}
\begin{proof}
Let us assume, by forcing if needed, that $\GCH$ holds in the ground model. Let $\mathbb{L}$ be the Easton support iteration of $\Add(\alpha^+, 1)$ for all inaccessible $\alpha \leq \kappa$. Let $V^{\mathbb{L}}$ be the generic extension. 

\begin{lemma}
In $V^{\mathbb{L}}$, $\kappa$ is $\kappa^+$-$\Pi^1_1$-subcompact. Moreover, this remains true after further forcing with $\Add(\kappa^{+},1)$.
\end{lemma}
\begin{proof}
Work in $V$. Let $\dot{A}$ be a name for a subset of $\kappa^{+}$. Let $\Phi$ be a $\Pi^1_1$-statement with parameter $\dot{A}$ which is true in the generic extension. Thus, the following $\Pi^1_1$-statement holds in the structure $\langle H(\kappa^+), \in, \mathbb{L}, \dot{A}, \Vdash_{\mathbb{L}}\rangle$: 
\[\forall \dot{X} \subseteq \mathbb{L}\times H(\kappa^{+}), \Vdash_{\mathbb{L}} \varphi(\dot{X}, \dot{A})\]
where $\varphi$ is a first order statement in the language of forcing. 

Since $\kappa$ is $\kappa^+$-$\Pi^1_1$-subcompact in $V$, we can find some cardinal $\rho < \kappa$, and $\tilde{\mathbb{L}}, \tilde{\dot{A}}$ such that there is an elementary embedding:
\[j\colon \langle H(\rho^{+}), \in, \tilde{\mathbb{L}}, \tilde{\dot{A}}, \Vdash_{\tilde{\mathbb{L}}}\rangle \to 
\langle H(\kappa^+), \in, \mathbb{L}, \dot{A}, \Vdash_{\mathbb{L}}\rangle.\] 
It is clear that $\tilde{\mathbb{L}} = \mathbb{L} \restriction \rho^{++}$. 

Let $p$ be a condition in $\mathbb{L}$. Without loss of generality, $p\in \im j$, and let $q\in\tilde{\mathbb{L}}$ such that $p = j(q)$. Let $\tilde{G}$ be a generic filter for $\tilde{\mathbb{L}}$ that contains $q$. Let $m = \bigcup_{r\in \tilde{G}} j(r(\rho))$, namely the union over the last coordinate of the $j$-image of all conditions in $\tilde{G}$. By the directed closure of the forcing $\Add(\kappa^{+},1)$, $m$ is a condition. Let $G$ be a generic that contains $\tilde{G} \restriction \rho^{++}$ and $m$. Note that $p\in G$. By Silver's criterion, $j$ extends to an elementary embedding between $H(\rho^{+})$ and $H(\kappa^{+})$ of the generic extension. Since we assumed that $q$ forces that $\Phi$ holds at $H(\rho^{+})$ and that $p$ forces that $\Phi$ holds at $H(\kappa^{+})$, the conclusion of the lemma follows.
\end{proof}

\end{proof}

\begin{theorem}
Simultaneous reflection for finite collections of stationary subsets of $\aleph_{\omega+1}$ is consistent relative to a cardinal $\kappa$ which is $\kappa^+$-$\Pi^1_1$-subcompact.
\end{theorem}
\begin{proof}
By Lemma \ref{lemma: pi11-subcompact is measurable} and Lemma \ref{lemma:
reflection at pi11 subcompact}, $\kappa$ is a measurable cardinal and every
collection of fewer than $\kappa$ many stationary subsets of
$S^{\kappa^{+}}_{<\kappa}$ reflects simultaneously at arbitrarily high cofinalities below $\kappa$. By Lemma \ref{lemma: indestructible pi11 subcompact}, this property of $\kappa$ can be forced to be indestructible under the forcing $\Add(\kappa^{+},1)$. By standard arguments, we may assume that $S^{\mu}_{\theta} \in I[\mu]$ for every $\theta < \mu < \kappa$ regular. Finally, by applying Theorem \ref{theorem: preserving stationary reflection} (with $S = S^{\kappa^{+}}_{<\kappa}$), the conclusion holds.
\end{proof}

\begin{remark}  In the model for the main theorem, there is a very good scale of
length $\kappa^+$ by Theorem 20 of \cite{CummingsForemanMagidor2001}.  By
Theorem 5 of \cite{CummingsForemanMagidor2001}, it
follows that simultaneous reflection for countable collections of stationary
sets fails in the final model in a strong way. \end{remark}

By using Theorem \ref{bounded reflection at subcompact}, if there is a measurable subcompact $\kappa$ and $\eta < \kappa$ then there is a generic extension in which every stationary subset of $\kappa^{+}$ of cofinality $<\eta$ reflects. In this model, we obtain that any stationary subset reflects at ordinals of arbitrary high cofinality. We may also assume that the approachability holds everywhere below $\kappa$. By the proof of Theorem \ref{bounded reflection at subcompact}, in this model, we obtain that stationary reflection for stationary subsets of $S^{\kappa^{+}}_{<\eta}$ is indestructible under the forcing $\Add(\kappa^{+},1)$. Thus, we conclude:
\begin{theorem}\label{theorem: bounded stationary reflection from measurable subcompact}
Let $\kappa$ be a measurable subcompact cardinal and let $n < \omega$. There is a generic extension in which every stationary subset of $S^{\aleph_{\omega+1}}_{\leq \aleph_n}$ reflects.
\end{theorem}

It is interesting to compare Theorem \ref{theorem: bounded stationary reflection from measurable subcompact} to Zeman's Theorem on the upper bound for the consistency strength of the failure of $\square_{\aleph_\omega}$, \cite{Zeman2017}.

By Theorem \ref{theorem: preserving stationary reflection}, the consistency of stationary reflection at the successor of a singular cardinal is bounded from above by the consistency of mildly indestructible stationary reflection at the successor of a measurable cardinal. 
\begin{question}
Assume that $\Refl(S)$ holds for some stationary subset of $\aleph_{\omega + 1}$. Is there an inner model with a measurable subcompact cardinal?
\end{question}
\section{Acknowledgments}
We would like to thank Thomas Gilton for carefully reading this paper and providing many helpful comments. We would like to thank Moti Gitik and Menachem Magidor for many useful discussions.

\providecommand{\bysame}{\leavevmode\hbox to3em{\hrulefill}\thinspace}
\providecommand{\MR}{\relax\ifhmode\unskip\space\fi MR }
\providecommand{\MRhref}[2]{%
  \href{http://www.ams.org/mathscinet-getitem?mr=#1}{#2}
}
\providecommand{\href}[2]{#2}


\begin{thebibliography}{10}

\bibitem{Baumgartner}
James~E Baumgartner, \emph{A new class of order types}, Annals of Mathematical
  Logic \textbf{9} (1976), no.~3, 187--222.

\bibitem{Bukovsky1973}
Lev Bukovsk{\`y}, \emph{Characterization of generic extensions of models of set
  theory}, Fundamenta Mathematicae \textbf{1} (1973), no.~83, 35--46.

\bibitem{Bukovsky1977}
Lev Bukovsk\'{y}, \emph{Iterated ultrapower and prikry's forcing},
  Commentationes Mathematicae Universitatis Carolinae \textbf{018} (1977),
  no.~1, 77--85 (eng).

\bibitem{CummingsWoodin}
James Cummings, \emph{A model in which {GCH} holds at successors but fails at
  limits}, Transactions of the American Mathematical Society \textbf{329}
  (1992), no.~1, 1--39.

\bibitem{CummingsForemanMagidor2001}
James Cummings, Matthew Foreman, and Menachem Magidor, \emph{Squares, scales
  and stationary reflection}, J. Math. Log. \textbf{1} (2001), no.~1, 35--98.
  \MR{1838355 (2003a:03068)}

\bibitem{Dehornoy1978}
Patrick Dehornoy, \emph{Iterated ultrapowers and prikry forcing}, Annals of
  Mathematical Logic \textbf{15} (1978), no.~2, 109 -- 160.

\bibitem{Faubion2012}
Zachary Faubion, \emph{Improving the consistency strength of stationary set
  reflection at $\aleph_{omega+ 1}$}, University of California, Irvine, 2012.

\bibitem{Gitik2010}
Moti Gitik, \emph{Prikry-type forcings}, pp.~1351--1447, Springer Netherlands,
  Dordrecht, 2010.

\bibitem{HarringtonShelah1985}
Leo Harrington and Saharon Shelah, \emph{Some exact equiconsistency results in
  set theory.}, Notre Dame Journal of Formal Logic \textbf{26} (1985), no.~2,
  178--188.

\bibitem{HayutLambieHanson2016}
Yair Hayut and Chris Lambie-Hanson, \emph{Simultaneous stationary reflection
  and square sequences}, arXiv preprint arXiv:1603.05556 (2016).

\bibitem{Ishiu2002}
Tetsuya Ishiu and Yasuo Yoshinobu, \emph{Directive trees and games on posets},
  Proceedings of the American Mathematical Society \textbf{130} (2002), no.~5,
  1477--1485.

\bibitem{JensenSteel2013}
Ronald Jensen and John Steel, \emph{K without the measurable}, The Journal of
  Symbolic Logic \textbf{78} (2013), no.~3, 708–734.

\bibitem{Magidor1982}
Menachem Magidor, \emph{Reflecting stationary sets}, Journal of Symbolic Logic
  \textbf{47} (1982), no.~4, 755–771.

\bibitem{Magidor2012Square}
Menachem Magidor and Chris Lambie-Hanson, \emph{On the strengths and weaknesses
  of weak squares}, Appalachian Set Theory 2006--2012 (2012), 301--330.

\bibitem{MagidorShelah1994}
Menachem Magidor and Saharon Shelah, \emph{When does almost free imply
  free?(for groups, transversals, etc.)}, Journal of the American Mathematical
  Society (1994), 769--830.

\bibitem{NeemanSteelSubcompact}
Itay Neeman and John Steel, \emph{Equiconsistenies at subcompact cardinals},
  submitted.

\bibitem{SchimmerlingZeman2001}
Ernest Schimmerling and Martin Zeman, \emph{Square in core models}, Bulletin of
  Symbolic Logic \textbf{7} (2001), no.~3, 305–314.

\bibitem{Shelah2013}
S.~Shelah, \emph{On incompactness for chromatic number of graphs}, Acta Math.
  Hungar. \textbf{139} (2013), no.~4, 363--371.

\bibitem{Todorcevic1983}
S.~Todor\v{c}evi\'c, \emph{On a conjecture of {R}. {R}ado}, J. London Math.
  Soc. (2) \textbf{27} (1983), no.~1, 1--8.

\bibitem{Zeman2017}
Martin Zeman, \emph{Two upper bounds on consistency strength of
  $\neg\square_{\aleph_\omega}$ and stationary set reflection at two successive
  $\aleph_n$}, Notre Dame Journal of Formal Logic (2017).

\end{thebibliography}
\end{document}